\documentclass[12pt]{article}

\usepackage{amssymb,graphicx,amscd,setspace,verbatim,amsmath,color,algpseudocode,algorithm, amsthm, multirow, multicol,booktabs}
\usepackage{algpseudocode,algorithm}
\usepackage{amssymb,amsmath,amsfonts,amsopn}
\usepackage{booktabs}
\usepackage{color}
\usepackage{graphicx}
\usepackage{multirow}
\usepackage[disable]{todonotes}

\newtheorem{lemma}{Lemma}

\newtheorem{proposition}{Proposition}

\DeclareMathOperator*{\argmin}{argmin}

\newcommand{\braces}[1]{\left\{ #1 \right \}}

\newcommand{\reals}{\mathbb{R}}

\newcommand{\norm}[1]{\left \| #1 \right \|}
\newcommand{\conv}{\text{conv}}

\newcommand{\Proj}{\text{Proj}}

\newcommand{\s}{s}
\newcommand{\z}{{z}}
\newcommand{\optval}{\zeta}

\newcommand{\K}{K}
\newcommand{\F}{D}
\newcommand{\Q}{\mathcal{Q}}
\newcommand{\SC}{\mathcal{S}}

\newcommand{\red}[1]{\textcolor{black}{#1}}
\newcommand{\blue}[1]{\textcolor{black}{#1}}
\definecolor{darkgreen}{rgb}{0,0.5,0}
\newcommand{\green}[1]{\textcolor{black}{#1}}

\newcommand{\st}{\; : \;}

\newcommand{\exclude}[1]{}
\newcommand{\bxi}{{\mathbf \xi}}

\newcommand{\dm}{\displaystyle}
\newcommand{\V}{V}

\begin{document}

\title{{Combining Progressive Hedging with a Frank-Wolfe Method to Compute Lagrangian Dual Bounds in Stochastic Mixed-Integer Programming}\thanks{The work of authors Boland, Christiansen, Dandurand, Eberhard, Linderoth, and Oliveira was supported in part or in whole by the Australian Research Council (ARC) grant ARC DP140100985. The work of authors Linderoth and Luedtke was supported in part by the U.S. Department of Energy, Office of Science, Office of Advanced Scientific Computing Research, Applied Mathematics program under contract number DE-AC02-06CH11357. }}

\author{
	Natashia Boland\thanks{Georgia Institute of Technology, Atlanta, Georgia, USA ({natashia.boland@gmail.com})}
	\and
	Jeffrey Christiansen\thanks{RMIT University, Melbourne, Victoria,
		Australia ({s3507717@student.rmit.edu.au}, {brian.dandurand@rmit.edu.au}, {andy.eberhard@rmit.edu.au}, {fabricio.oliveira@rmit.edu.au})}
	\and
	Brian Dandurand\footnotemark[3]
	\and
	Andrew Eberhard\footnotemark[3]
	\and
	Jeff Linderoth\thanks{Department of Industrial and Systems
		Engineering, Wisconsin Institutes of Discovery, University of
		Wisconsin-Madison, Madison, Wisconsin, USA ({linderoth@wisc.edu},
		{jim.luedtke@wisc.edu})}
	\and
	James Luedtke\footnotemark[4]
	\and
	Fabricio Oliveira\footnotemark[3]
}
\date{20/05/2016}
\maketitle

\begin{abstract}
We present a new primal-dual algorithm for computing the value of the Lagrangian dual of a stochastic mixed-integer program (SMIP) formed by relaxing its nonanticipativity constraints. \red{This dual is widely used in decomposition methods for the solution of SMIPs.} The algorithm relies on the well-known progressive hedging method, but unlike previous progressive hedging approaches for SMIP, our algorithm can be shown to converge to the optimal Lagrangian dual value. The key improvement in the new algorithm is an inner loop of optimized linearization steps, similar to those taken in the classical Frank-Wolfe method. Numerical results demonstrate that our new algorithm empirically outperforms the standard implementation of progressive hedging for obtaining bounds in SMIP. \newline
\bigskip

\end{abstract}




\section{Introduction}
\label{SectionIntro}

Stochastic programming with recourse provides a framework for modeling problems where decisions are made in stages.  Between stages, some uncertainty in the problem parameters is unveiled, and decisions in subsequent stages may depend on the outcome of this uncertainty. When some decisions are modeled using discrete variables, the problem is known as a Stochastic Mixed-Integer Programming (SMIP) problem. The ability to simultaneously model uncertainty and discrete decisions make SMIP a powerful modeling paradigm for applications. Important applications employing SMIP models include unit commitment and hydro-thermal generation scheduling \cite{NowakRomisch:2000,takriti.birge.long:96}, military operations \cite{salmeron.wood.morton:09}, vaccination planning \cite{ozaltin.et.al:11,tanner.sattenspiel.ntaimo:08}, air traffic flow management \cite{agustin.et.al:12}, forestry management and forest fire response \cite{badilla.et.al:15,ntaimo.et.al:12}, and supply chain and logistics planning \cite{laporte.louveaux.mercure:92,louveaux:86}. However, the combination of uncertainty and discreteness makes this class of problems extremely challenging from a computational perspective. In this paper, we present a new and effective algorithm for computing lower bounds that arise from a Lagrangian-relaxation approach.

The mathematical statement of a two-stage SMIP is
\begin{equation}\label{EqSMIP}
\optval^{SMIP} := \min_{x} \braces{c^\top x + \Q(x) \st x \in X},
\end{equation}
where the vector $c \in \mathbb{R}^{n_x}$ is known, and $X$ is a mixed-integer linear set consisting of linear constraints and integer restrictions on some components of $x$.  The function $\Q : \reals^{n_x} \mapsto \reals$ is the expected recourse value
\[ \Q(x) := \mathbb{E}_\bxi \left[ \min_y \braces{ q(\bxi)^\top y \st W(\bxi) y = h(\bxi) - T(\bxi)x, y \in Y(\bxi)} \right].
\]
We assume that the random variable $\bxi$ is taken from a discrete distribution indexed by the finite set $\SC$\red{,} consisting \blue{of the} realizations, $\xi_1, \dots, \xi_{|\SC|}$, with strictly positive corresponding probabilities of realization, $p_1, \dots, p_{|\SC|}$. When $\bxi$ is not discrete, a finite scenario approximation can be obtained via Monte Carlo sampling \cite{kleywegt.shapiro.homemdemello:01,mak.morton.wood:99} or other methods \cite{homem-de-mello:08,chen.et.al:14}.  Each realization $\xi_\s$ of $\bxi$, is called a {\em scenario} and encodes the realizations observed for each of the random elements $\left(q(\xi_\s), h(\xi_\s), W(\xi_\s), T(\xi_\s), Y(\xi_\s)\right)$. For notational brevity, we refer to this collection of random elements respectively as $\left(q_\s, h_\s, W_\s, T_\s, Y_\s \right)$.  \exclude{We assume that the random elements have the   dimensions $q_\s \in \mathbb{R}^{n_y}$, $h_\s \in \mathbb{R}^{m}$, $W_\s \in \mathbb{R}^{n_y \times m}$, and $T_\s \in \mathbb{R}^{n_x \times m}$ for each $s \in \SC$.} For each $\s\in \SC$, the set $Y_\s \subset \reals^{n_y}$ is a mixed-integer set containing both linear constraints and integrality constraints on a subset of the
variables, $y_s$.

The problem \eqref{EqSMIP} may be reformulated as its \textit{deterministic equivalent}
\begin{equation}\label{EqSMIPEF}
\optval^{SMIP} = \min_{x,y} \braces{\begin{array}{c} c^\top x + \dm\sum_{\s \in \SC} p_\s q_\s^\top y_\s \st
(x,y_\s)  \in \K_\s, \forall \s \in \SC\end{array}},
\end{equation}
where $\K_s := \braces{(x, y_s) \st W_\s y_\s = h_\s - T_\s x, x \in
  X, y_\s \in Y_\s}$.
\green{Problem~\eqref{EqSMIPEF} has a special structure that} can be algorithmically exploited by decomposition methods. To induce a decomposable structure, scenario-dependent copies $x_\s$ for each $\s \in S$ of the first-stage variable $x$ are introduced to create the following reformulation of \eqref{EqSMIPEF}: %
\begin{equation}\label{EqSMIPEFS}
\optval^{SMIP} = \min_{x,y,\z} \braces{\begin{array}{c}\!\!\!\dm\sum_{\s \in \SC} p_\s(c^\top x_\s + q_\s^\top y_\s) \st (x_\s,y_\s) \in \K_\s, x_\s = \z,\; \forall \s \in \SC, z \in \reals^{n_x} \end{array}\!\!\!}.
\end{equation}
The constraints $x_\s = \z$, $\s \in \SC$, enforce nonanticipativity for first-stage decisions; the first-stage decisions $x_\s$ must be the same $(\z)$ for each scenario $\s \in \SC$. Applying Lagrangian relaxation to the nonanticipativity constraints in problem~\eqref{EqSMIPEFS} yields the \emph{nonanticipative Lagrangian dual function}
\begin{equation}\label{EqLagRelMu}
\phi(\mu) := \min_{x,y,\z}\braces{\begin{array}{c} \sum_{s \in \SC} \left\lbrack p_\s(c^\top x_\s + q_\s^\top y_\s) + \mu_\s^\top (x_\s - \z)\right\rbrack \st \\ (x_\s,y_\s) \in
  \K_\s, \ \forall \s \in \SC, z \in \reals^{n_x}\end{array}},
\end{equation}
where $\mu=(\mu_1,\dots,\mu_{|\SC|}) \in \prod_{\s \in \SC} \reals^{n_x}$ is the vector of multipliers associated with the relaxed constraints $x_\s = \z$, $\s \in \SC$. By setting $\omega_\s := \frac{1}{p_\s} \mu_\s$,~\eqref{EqLagRelMu} may be rewritten as
\begin{equation}\label{EqLagRel}
\phi(\omega) := \min_{x,y,\z}\braces{\sum_{s \in \SC} p_\s L_s(x_s,y_s,\z,\omega_s) \st (x_\s,y_\s) \in \K_\s,\;  \forall \s \in \SC, \; z \in \reals^{n_x}},
\end{equation}
where
\[ L_s(x_s,y_s,\z,\omega_s):= c^\top x_\s + q_\s^\top y_\s+ \omega_\s^\top (x_\s - \z). \]
Since $\z$ is unconstrained in the optimization problem in the definition \eqref{EqLagRel},
in order for the Lagrangian function $\phi(\omega)$ to be bounded from below, we require as a condition of dual feasibility that $\sum_{\s \in \SC} p_\s\omega_\s = 0$.
Under this assumption, the $\z$ term vanishes, and the Lagrangian dual function~\eqref{EqLagRel} decomposes into separable functions,
\begin{equation}
\label{eq:lagrangian-1}
\phi(\omega) = \sum_{\s \in \SC} p_\s\phi_\s(\omega_s),
\end{equation}
%
where for each $\s \in \SC$,
\begin{equation}
\label{eq:lagrangian-1s}
\phi_\s(\omega_s) := \min_{x,y} \braces{ (c+\omega_\s)^\top x + q_\s^\top y   \st  (x,y) \in \K_\s}.
\end{equation}
The reformulation~\eqref{eq:lagrangian-1} is the basis for parallelizable approaches for computing dual bounds that are used, for example, in the dual decomposition methods developed in~\cite{caroe1999dual, LubinEtAl2013}.

For any choice of $\omega = \left(\omega_1, \dots, \omega_{|\SC|}\right)$, it is well-known that the value of the Lagrangian provides a lower bound on the optimal solution to \eqref{EqSMIP}: $\phi(\omega) \leq \optval^{SMIP}$.
The problem of finding the best such lower bound is the \emph{Lagrangian dual problem}:
\begin{equation}
\label{eq:zldprob}
 \optval^{LD} := \sup_{ \omega }
\braces{ \phi(\omega) \st \sum_{\s \in \SC} p_\s\omega_\s = 0}.
\end{equation}
The primary contribution of this work is a new and effective method for
solving \eqref{eq:zldprob}, thus enabling a practical and efficient computation of high-quality
lower bounds for $\optval^{SMIP}$.


The function $\phi(\omega)$ is a piecewise-\red{affine} concave function,
and many methods are known for maximizing such functions. These
methods include the subgradient method~\cite{Shor1985},
the augmented Lagrangian (AL) method~\cite{Hestenes1969,Powell1969},
and the alternating direction method of multipliers
(ADMM)~\cite{GabayMercier1976,EcksteinBertsekas1992,BoydEtAl2011}. The
subgradient method has mainly theoretical significance, since it is
difficult to develop reliable and efficient step-size rules for the
dual variables $\omega$ (see, e.g., Section 7.1.1
of~\cite{Ruszczynski2006}). As iterative primal-dual approaches,
methods based on the AL method or ADMM are more effective in practice.
However, in the context of SMIP, both methods require convexification
of the constraints $\K_\s$, $\s \in \SC$ to have a meaningful
theoretical support for convergence to the best lower bound value
$\optval^{LD}$.  Furthermore, both methods require the solution of
additional mixed-integer linear programming (MILP) subproblems in
order to recover the Lagrangian lower bounds associated with the dual
values, $\omega$ \cite{GadeEtAl2016}.  ADMM has a more straightforward
potential for decomposability and parallelization than the AL method,
and so in this work, we develop a theoretically-supported
modification of a method based on ADMM.




When specialized to the deterministic equivalent problem~\eqref{EqSMIPEF} \blue{in the context of stochastic programming}, ADMM is referred to as Progressive Hedging (PH)~\cite{RockafellarWets1991PH,WatsonWoodruff2011PH}. When the sets $\K_\s$, $\s \in \SC$, are convex, the limit points of the sequence of solution-multiplier pairs $\braces{((x^k,y^k,\z^k),\omega^k)}_{k=1}^\infty$ generated by PH are saddle points of the deterministic equivalent
problem~\eqref{EqSMIPEF}, whenever such saddle points exist. When the constraints $(x_\s,y_\s) \in \K_\s$, $\s \in \SC$, enforce nontrivial mixed-integer restrictions,
the set $\K_\s$ is not convex and
PH becomes a heuristic approach with no guarantees of convergence~\cite{LokketangenWoodruff1996}. Nevertheless, some measure of success\red{, in practice,} has been observed in~\cite{WatsonWoodruff2011PH} while applying PH to problems of the form~\eqref{EqSMIPEFS}.  More recently,~\cite{GadeEtAl2016} showed that valid Lagrangian lower bounds can be calculated from the iterates of the PH algorithm when the sets $\K_\s$ are not convex. However, their implementation of the algorithm does not offer any guarantee that the lower bounds will converge to the optimal value $\optval^{LD}$.  Moreover, additional computational effort, in solving additional MILP subproblems, must be expended, in order to compute the lower bound. Our contribution is to extend the PH-based approach in~\cite{GadeEtAl2016}, creating an algorithm whose lower bound values converge to $\optval^{LD}$\red{, in theory,} and for which lower bound calculations do not require additional computational effort. Computational results in Section~\ref{Sec4} demonstrate that the new method outperforms the existing PH-based method, in terms of both quality of bound and efficiency of computation.

To motivate our approach, we first consider the application of PH to the following well-known primal characterization of $\optval^{LD}$:
\begin{equation}\label{EqSMIPEFSConv}
  \optval^{LD} = \min_{x,y,\z} \braces{\sum_{\s \in \SC} p_\s(c^\top x_\s + q_\s^\top y_\s) \st (x_\s,y_\s) \in \conv(\K_\s), \;
    \exclude{\forall \s \in \SC, } x_\s = \z,\; \forall \s \in \SC},
\end{equation}
where $\conv(\K_\s)$ denotes the convex hull of $\K_\s$ for each $\s \in \SC$. (See, for example, Theorem 6.2 of~\cite{NemhauserWolsey1988}.)
The sequence of Lagrangian bounds $\braces{\phi(\omega^k)}$ generated
by the application of PH to~\eqref{EqSMIPEFSConv} is known to be
convergent. Thus, the value of the Lagrangian dual, $\optval^{LD}$, may, in theory, be computed by applying PH to~\eqref{EqSMIPEFSConv}. However, in practice, an explicit polyhedral description of $\conv(\K_\s)$, $\s \in \SC$ is, generally, not available, thus raising the issue of implementability.

The absence of such an explicit description motivates an application of a solution approach to the PH primal update step that iteratively constructs an improved inner approximation of each $\conv(\K_\s)$, $\s \in \SC$.
For this purpose, we apply a solution approach to the PH primal update
problem that is based on the Frank-Wolfe (FW)
method~\cite{FrankWolfe1956}.
Our approach has the additional benefit of
providing Lagrangian bounds at no additional computational cost.

One simple, theoretically-supported integration of a FW-like method and
PH is realized by having the PH primal updates computed using a method
called the Simplicial Decomposition Method
(SDM)~\cite{Holloway1974,Hohenbalken1977}.  
SDM is an extension of the FW method that makes use of progressively-improving inner approximations to each set $\conv(\K_\s)$, $\s \in
\SC$. The finite optimal convergence of each application of SDM
follows directly from the polyhedral structure
$\conv(\K_\s)$, and the (practically reasonable) assumption that $\conv(\K_\s)$ is bounded for each $\s \in \SC$.
\todo[inline]{Where do we assume/state that $\conv(K_s)$ is bounded?  Proof of convergence in Proposition 1 relies on boundedness. (I thought we just needed it to be polyhedral, but I am happy to assume bounded for this paper)}

For computing improvements in the Lagrangian bound efficiently,
convergence of SDM to the optimal solution of the subproblem is too
costly and not necessary.  We \red{thus} develop a modified integration whose
theoretically-supported convergence analysis is based not on the optimal convergence of SDM, but rather on its ability to adequately extend the inner approximations of each $\conv(\K_\s)$, $\s \in \SC$.
The main contribution of this paper is the development, convergence
analysis, and application of a new algorithm, called FW-PH, which is used to compute high-quality Lagrangian bounds for SMIPs efficiently and with a high potential for parallelization. FW-PH is efficient in that, under mild assumptions, each dual update and Lagrangian bound computation may be obtained by solving, for each $\s \in \SC$, just one MILP problem and one continuous \red{convex} quadratic problem.
In contrast, each dual update of PH requires the solution of a mixed-integer quadratic programming (MIQP) subproblem for each $\s \in \SC$, and each PH Lagrangian bound computation requires the solution of one MILP subproblem for each $\s \in \SC$. In our convergence analysis, conditions are provided under which the sequence of Lagrangian bounds generated by FW-PH converges to the optimal Lagrangian bound $\optval^{LD}$. To the best of our knowledge,
the combination of PH and FW in a manner that is theoretically supported, computationally efficient, and parallelizable is new, in spite of the convergence analyses of both PH and FW being well-developed.

This paper is organized as follows. In Section \ref{Sec2}, we present
the theoretical background of PH and a brief technical lemma regarding
the inner approximations generated by SDM; this background is
foundational for the proposed FW-PH method. In Section 3, we present
the FW-PH method and a convergence analysis.
The results of numerical experiments comparing the Lagrangian bounds computed with PH and those with FW-PH are presented in Section \ref{Sec4}.
We conclude in Section \ref{Sec5} with a discussion of the results obtained and with suggested directions for further research.

\section{Progressive Hedging and Frank-Wolfe-Based Methods} \label{Sec2}

The \linebreak
Augmented Lagrangian (AL) function based on the relaxation of the nonanticipativity constraints $x_\s = \z$, $\s \in \SC$, is
$$L^\rho(x,y,\z,\omega):=\sum_{\s \in \SC} p_\s L_\s^\rho(x_s,y_s,\z,\omega_\s),$$
where
$$L_\s^\rho(x_s,y_s,\z,\omega_\s):=c^\top x_\s + q_\s^\top y_\s+ \omega_\s^\top (x_\s - \z) + \frac{\rho}{2}\norm{x_\s - \z}_2^2$$
and $\rho>0$ is a penalty parameter.  By changing the feasible region, denoted here by $\F_\s$, $\s\in\SC$, the Augmented Lagrangian can be used in a Progressive Hedging approach to solve either problem~\eqref{EqSMIPEFS} or problem~\eqref{EqSMIPEFSConv}.  Pseudocode for the PH algorithm is given in Algorithm~\ref{AlgPH}.

\begin{algorithm}[hbtp]
\caption{PH applied to problem~\eqref{EqSMIPEFS} ($\F_\s=\K_s$\exclude{, $\s \in \SC$}) or~\eqref{EqSMIPEFSConv} ($\F_\s=\conv(\K_\s)$\exclude{, $\s \in\SC$}).  \label{AlgPH}}
\begin{algorithmic}[1]
\State Precondition: $\sum_{\s \in \SC} p_\s \omega_\s^{0} = 0$
\Function{PH}{$\omega^0$, $\rho$, $k_{max}$, $\epsilon$}
 \For{$\s \in \SC$} \label{FWPH0InitBeginPH}
    \State $({x}_\s^0,{y}_s^0) \in \argmin_{x,y} \braces{(c+\omega_\s^0)^\top x + q_\s^\top y \st (x,y) \in \F_\s}$
  \EndFor
   \State $\phi^0 \gets \sum_{\s \in \SC} p_\s\left[ (c+\omega_\s^0)^\top x_\s^0 + q_\s^\top y_\s^0 \right]$
    \State $\z^0 \gets  \sum_{\s \in \SC} p_\s x_\s^0$
    \State $\omega_\s^{1} \gets \omega_\s^{0} + \rho(x_\s^{0}-\z^{0})$ for all $\s \in \SC$ \label{FWPH0InitEndPH}
    \For{$k=1,\dots,k_{max}$}
	    \For {$\s \in \SC$}
    	       \State $\phi_\s^k \gets \min_{x,y} \braces{(c+\omega_\s^k)^\top x + q_\s^\top y \st (x,y) \in \F_\s}$ \label{ComputePhiPHS}
	        \State $(x_\s^k,y_\s^k) \in \argmin_{x,y} \braces{L_\s^{\rho}(x,y,\z^{k-1},\omega_\s^{k}) \st (x,y) \in \F_\s }$ \label{PHSP}
	    \EndFor
	    \State $\phi^k \gets \sum_{\s \in \SC} p_\s\phi_\s^k$ \label{ComputePhiPH}
	   \State $\z^k \gets \sum_{\s \in \SC} p_\s x^k_\s$
	   \If{$\sqrt{\sum_{\s \in \SC} p_\s \norm{x_\s^k-\z^{k-1}}_2^2} < \epsilon $}
		\State {\bf return} { $(x^{k},y^{k},\z^{k}, \omega^{k}, \phi^k)$}
	\EndIf
	\State $\omega_\s^{k+1} \gets \omega_\s^{k} + \rho(x_\s^{k}-\z^{k})$ for all $\s \in \SC$

    \EndFor
    \State {\bf return} { $(x^{k_{max}},y^{k_{max}},\z^{k_{max}}, \omega^{k_{max}}, \phi^{k_{max}})$}
\EndFunction
\end{algorithmic}
\end{algorithm}
In Algorithm~\ref{AlgPH}, $k_{max} > 0$ is the maximum number of iterations, and $\epsilon > 0$ parameterized the convergence tolerance.
The initialization of Lines~\ref{FWPH0InitBeginPH}--\ref{FWPH0InitEndPH} provides an initial target primal value $z^0$ and dual values $\omega_\s^1$, $\s \in \SC$, for the main iterations $k \ge 1$.
Also, an initial Lagrangian bound $\phi^0$ can be computed from this initialization.
For $\epsilon > 0$, the Algorithm~\ref{AlgPH} termination criterion $\sqrt{\sum_{\s \in \SC} p_\s \norm{x_\s^k-\z^{k-1}}_2^2} < \epsilon$
is motivated by the addition of the squared norms of the primal and dual residuals associated with problem~\eqref{EqSMIPEFS}. (See Section 3.3 of~\cite{BoydEtAl2011}.)
In summing the squared norm primal residuals $p_\s \norm{x_\s^k -  \z^{k}}_2^2$, $\s \in \SC$, and the squared norm dual residual $\norm{\z^k - \z^{k-1}}_2^2$, we have
\begin{equation}\sum_{\s \in \SC}p_\s \left[\norm{x_\s^k -  \z^{k}}_2^2 +\norm{\z^k - \z^{k-1}}_2^2\right]
=\sum_{\s \in \SC}  p_\s\norm{x_\s^k-\z^{k-1}}_2^2 \label{PDDiscr}
\end{equation}
The equality in~\eqref{PDDiscr} follows since, for each $\s \in \SC$, the cross term resulting from the expansion of the squared norm
$
\norm{(x_\s^k-\z^k) + (\z^k -\z^{k-1})}_2^2
 $
 vanishes; this is seen in the equality $\sum_{\s \in \SC} p_\s(x_\s^k - \z^{k})=0$ due to the construction of $z^k$.

The Line~\ref{ComputePhiPHS} subproblem of Algorithm~\ref{AlgPH} is an addition to the original PH algorithm. Its purpose is to compute Lagrangian bounds (Line~\ref{ComputePhiPH}) from the current dual solution $\omega^k$~\cite{GadeEtAl2016}.
Thus, the bulk of computational effort in Algorithm~\ref{AlgPH} applied to problem~\eqref{EqSMIPEFS} (the case with $\F_\s=\K_\s$) resides in
computing solutions to the MILP (Line~\ref{ComputePhiPHS}) and MIQP (Line~\ref{PHSP}) subproblems. \blue{Note that Line~\ref{ComputePhiPHS} may be omitted if the corresponding Lagrangian bound for $\omega^k$ is not desired.}

 \subsection{Convergence of PH}
 The following proposition addresses the convergence of PH applied to problem~\eqref{EqSMIPEFSConv}.
 \begin{proposition}\label{PropPH}
Assume that problem~\eqref{EqSMIPEFSConv} is feasible with $\conv(\K_\s)$ bounded for each $\s \in\SC$, and let Algorithm~\ref{AlgPH} be applied to problem~\eqref{EqSMIPEFSConv} (so that $\F_\s = \conv(\K_\s)$ for each $\s \in \SC$) with tolerance $\epsilon=0$ for each $k \ge 1$.
Then, the limit $\lim_{k \to \infty} \omega^k = \omega^*$ exists, and furthermore,
\begin{enumerate}
\item $\lim_{k \to \infty} \sum_{\s \in \SC} p_\s (c^\top x_\s^k + q_\s^\top y_\s^k) = \optval^{LD}$,
\item $\lim_{k\to\infty} \phi(\omega^k) = \optval^{LD}$,
\item $\lim_{k \to \infty} (x_s^k - z^k)=0$ for each $\s \in \SC$,
\end{enumerate}
and each limit point $(((x_\s^*,y_\s^*)_{\s \in \SC}, \z^*)$ is an optimal solution for~\eqref{EqSMIPEFSConv}.
 \end{proposition}
 \begin{proof}
   Since the constraint sets $\F_\s=\conv(\K_\s)$, $\s \in \SC$, are bounded, and problem~\eqref{EqSMIPEFSConv} is feasible, problem~\eqref{EqSMIPEFSConv} has an optimal solution $((x_\s^*,y_\s^*)_{\s \in \SC},z^*)$ with optimal value $\optval^{LD}$. The feasibility of problem~\eqref{EqSMIPEFSConv}, the linearity of its objective function, and the bounded polyhedral structure of its constraint set $\F_\s=\conv(\K_\s)$, $\s \in \SC$, imply that
   the hypotheses for PH convergence to the optimal solution are met (See Theorem 5.1 of~\cite{RockafellarWets1991PH}). Therefore, $\braces{\omega^k}$ converges to some $\omega^*$, $\lim_{k \to \infty} \sum_{\s \in \SC} p_\s (c^\top x_\s^k + q_\s^\top y_\s) = \optval^{LD}$, $\lim_{k\to\infty} \phi(\omega^k) = \optval^{LD}$, and $\lim_{k \to \infty} (x_s^k - z^k)=0$ for each $\s \in \SC$ all hold. The boundedness of each $\F_\s = \conv(\K_\s)$, $\s \in \SC$, furthermore implies the existence of limit points $((x_\s^*,y_\s^*)_{\s \in \SC},z^*)$ of $\braces{((x_\s^k,y_\s^k)_{\s \in \SC},z^k)}$, which are optimal solutions for~\eqref{EqSMIPEFSConv}.
 \end{proof}
Note that the convergence in Proposition~\ref{PropPH} applies to the continuous problem~\eqref{EqSMIPEFSConv} but {\it not} to the mixed-integer problem~\eqref{EqSMIPEFS}.  In problem~\eqref{EqSMIPEFS}, the constraint sets $\K_\s$, $\s \in \SC$, are not convex, so there is no guarantee that Algorithm~\ref{AlgPH} will converge when applied to~\eqref{EqSMIPEFS}. However, the application of PH to problem~\eqref{EqSMIPEFSConv} requires\red{, in  Line~\ref{PHSP},} the optimization of the Augmented Lagrangian over the sets $\conv(\K_\s)$, $\s \in \SC$, for which an explicit linear description is unlikely to be known.  In the next section, we demonstrate how to circumvent this difficulty by constructing inner approximations of the polyhedral sets $\conv(\K_\s)$, $\s \in \SC$.

\subsection{A Frank-Wolfe approach based on Simplicial Decomposition}

To use Algorithm~\ref{AlgPH} to solve \eqref{EqSMIPEFSConv} requires a method for solving the subproblem
\begin{equation}
\label{eq:quadch}
(x_\s^k,y_\s^k) \in \argmin_{x,y} \braces{ L_\s^{\rho}(x,y,\z^{k-1},\omega_\s^{k}) : (x,y) \in \conv(\K_\s) }
\end{equation}
appearing in Line \ref{PHSP} of the algorithm. Although an explicit description of $\conv(\K_\s)$ is not readily
available, if we have a {\it linear} objective function, then we can replace $\conv(\K_\s)$ with $\K_\s$. This motivates
the application of a FW algorithm for solving \eqref{eq:quadch}, since the FW algorithm solves a sequence of
problems in which the nonlinear objective is linearized using a first-order approximation.

The Simplicial Decomposition Method (SDM) is an extension of the FW method,
where the line searches of FW are replaced by searches over polyhedral inner approximations.
SDM can be applied to solve a
feasible, bounded problem of the general form
\begin{equation}\label{EqFWProblem}
\optval^{FW}:=\min_x \braces{f(x): x \in \F},
\end{equation}
{with nonempty compact convex set $\F$ and continuously differentiable convex function $f$}.
Generically, given a current solution $x^{t-1}$ and inner approximation $\F^{t-1} \subseteq \F$, iteration $t$ of the SDM consists of solving
\[
\widehat{x} \in \argmin_{x} \braces{\nabla_x f(x^{t-1})^\top x  : x \in \F}
\]
updating the inner approximation as $\F^t \gets \conv(\F^{t-1} \cup \braces{\widehat{x}})$, and finally
choosing
\[ 	x^{t} \in \argmin_x \braces{f(x) : x \in \F^t} . \]
The algorithm terminates when the bound gap is small, specifically, when \red{$$\Gamma^t :=-\nabla_x f(x^{t-1})^\top (\widehat{x}-x^{t-1}) < \tau,$$} where $\tau \geq 0$ is a given tolerance.

%

The application of SDM to solve problem \eqref{eq:quadch}, i.e., to minimize $L_\s^{\rho}(x,y,\z,\omega_\s)$ over $(x,y) \in \conv(\K_\s)$, for a given $\s\in\SC$, is presented in Algorithm \ref{AlgSDM}.
Here, $t_{max}$ is the maximum number of iterations and $\tau > 0$ is a convergence tolerance. $\Gamma^t$ is the bound gap used to measure closeness to optimality,  and $\phi_s$ is used to compute a Lagrangian bound
as described in the next section. The inner approximation to $\conv(\K_\s)$ at iteration $t\geq 1$ takes the form $\conv(\V^t_\s)$, where $\V_\s^t$ is a finite set of points, with $\V^t_\s\subset\conv(\K_\s)$. The points added by Algorithm~\ref{AlgSDM} to the initial set, $\V^0_\s$, to form $\V^t_\s$, are all in $\K_\s$: here $\mathcal{V}(\conv(\K_\s))$ is the set of extreme points of $\conv(\K_\s)$ and, of course, $\mathcal{V}(\conv(\K_\s))\subseteq K_s$.

\begin{algorithm}[h]
\caption{SDM applied to problem \eqref{eq:quadch}.\label{AlgSDM}}
\begin{algorithmic}[1]
\State Precondition: $\V^0_\s\subset\conv(\K_\s)$ and $z=\sum_{\s \in \SC} p_\s x_\s^0$
\Function{${\rm SDM}$}{$\V_s^0$, $x_s^0$, $\omega_s$, $z$, $t_{max}$, $\tau$}
       \For{$t=1,\dots,t_{max}$}
		 \State $\widehat{\omega}^t_s \gets  \omega_s + \rho(x_s^{t-1} - z)$ \label{SDMomega}
	\State $(\widehat{x}_s,\widehat{y}_s) \in \argmin_{x,y} \braces{(c + \widehat{\omega}_s^t)^{\top} x + q_\s^\top y \st
	(x,y) \in  \mathcal{V}(\conv(\K_s))}$ \label{DirFindingSDM}
	\If{$t=1$}
		\State $\phi_s \gets (c + \widehat{\omega}_s^t)^{\top} \widehat{x}_s + q_\s^\top \widehat{y}_s$ \label{SDMEvalPhiS}
	\EndIf
	\State $\Gamma^t \gets -[(c + \widehat{\omega}_\s^t)^{\top}(\widehat{x}_\s-x_\s^{t-1}) + q_\s^{\top}(\widehat{y}_\s - y^{t-1}_\s)]$
	\State $\V_\s^t \gets\V_\s^{t-1} \cup \braces{(\widehat{x},\widehat{y})}$ \label{IAUpdate}
	
       	\State $(x^{t}_s,y^{t}_s) \in \argmin_{x,y} \braces{L_s^{\rho}(x,y,z,\omega_s) \st (x,y) \in
			\conv(\V_s^{t})}$ \label{SDM-XYUpdate}
       	\If{$\Gamma^t \le \tau$}
       		\State \textbf{return} { $\left(x_\s^{t}, y_\s^t, \V_\s^{t}, \phi_s \right)$}
       	\EndIf
        \EndFor
      \State \textbf{return} { $\left(x^{t_{max}}_\s, y^{t_{max}}_\s, \V_\s^{t_{max}}, \phi_\s \right)$}
\EndFunction
\end{algorithmic}
\end{algorithm}
\noindent

Observe that
$$
  \nabla_{(x,y)} L_\s^{\rho}(x,y,\z,\omega_\s)|_{(x,y)=\left(x^{t-1}_\s,y^{t-1}_\s\right)} = \left[\begin{array}{c}c+\omega_\s + \rho(x^{t-1}_\s - z) \\ q_\s\end{array}\right]= \left[\begin{array}{c} c+\widehat{\omega}_\s\\ q_\s\end{array}\right],
$$
and so the optimization at Line \ref{DirFindingSDM} is minimizing the gradient approximation to \linebreak
 $L_\s^{\rho}(x,y,\z,\omega_\s)$ at the point $(x^{t-1}_\s,y^{t-1}_\s)$. Since this is a linear objective function, optimization over $\mathcal{V}(\conv(\K_\s))$ can be accomplished by optimization over $\K_\s$ (see, e.g., \cite{NemhauserWolsey1988}, Section I.4, Theorem 6.3). Hence Line~\ref{DirFindingSDM} requires a solution of a single-scenario MILP.

The optimization at Line~\ref{SDM-XYUpdate} can be accomplished by expressing $(x,y)$ as a convex combination of the finite set of points, $\V^t_\s$, where the weights \green{$a\in \reals^{|\V^t_\s|}$} in the convex combination are now \green{also} decision variables. 
\green{That is, the Line~\ref{SDM-XYUpdate} problem is solved with a solution to the following convex continuous quadratic subproblem
\begin{equation}\label{QPConvexComb}
(x^{t}_s,y^{t}_s,a) \in \argmin_{x,y,a} \braces{\begin{array}{c} L_s^{\rho}(x,y,z,\omega_s) \st (x,y) = \sum_{(\widehat{x}^i,\widehat{y}^i) \in \V^t_\s} a_i  (\widehat{x}^i,\widehat{y}^i), \\ \sum_{i=1,\dots,|\V^t_\s|} a_i = 1, \;\text{and}\; a_i \ge 0 \;\text{for}\; i=1,\dots,|\V^t_\s| \end{array}}.
\end{equation}
For implementational purposes, the $x$ and $y$ variables may be substituted out of the objective of problem~\eqref{QPConvexComb}, leaving $a$ as the only decision variable, with the only constraints being nonnegativity of the $a$ components and the requirement that they sum to $1$.}

The Simplicial Decomposition Method is known to terminate finitely
with an optimal solution when $\F$ is polyhedral
\cite{Holloway1974}, so the primal update step Line~\ref{PHSP},
Algorithm~\ref{AlgPH} with $\F_\s = \conv(\K_\s)$ could be
accomplished with the SDM, resulting in an algorithm that converges to
a solution giving the Lagrangian dual bound $\optval^{LD}$.  However,
since each inner iteration of Line~\ref{DirFindingSDM},
Algorithm~\ref{AlgSDM} requires the solution of a MILP, using $t_{\max}$ large enough to ensure SDM terminates optimally is not efficient for our purpose of
computing Lagrangian bounds.  In the next section, we give an
adapation of the algorithm that requires the solution of only {\em one} MILP
subproblem per scenario at each major iteration of the PH algorithm.

\section{The FW-PH method} \label{Sec3}

In order to make the SDM efficient when used with PH to solve the
problem \eqref{EqSMIPEFSConv}, the minimization of the Augmented
Lagrangian can be done approximately.  This insight can greatly reduce
the number of MILP subproblems solved at each inner iteration and
forms the basis of our algorithm FW-PH.  Convergence of FW-PH relies
on the following lemma, which states an important expansion property
of the inner approximations employed by the SDM.

\begin{lemma}\label{FiniteSDMTermination}
For any scenario $\s \in \SC$ and iteration $k \ge 1$, let Algorithm~\ref{AlgSDM} be applied to the minimization problem \eqref{eq:quadch} for any $t_{max} \ge 2$. For $1 \le
t < t_{max}$,
if
\begin{equation}\label{LemmaSDMEq0}
({x}_s^t,y_s^t) \not\in \argmin_{x,y} \braces{L_s^{\rho}(x,y,\z^{k-1},\omega_\s^{k}) \st (x,y) \in \conv(\K_\s)}
\end{equation}
holds, then $\conv(\V_s^{t+1}) \supset \conv(\V_s^t)$.
\end{lemma}
\begin{proof}
For $\s \in \SC$ and $k \ge 1$ fixed, we know that by construction
$$({x}^t_s,y^t_s) \in \argmin_{x,y} \braces{L_\rho^s(x,y,\z^{k-1},\omega_\s^{k}) : (x,y) \in \conv(\V^t_s)}$$  for $t\geq 1.$
Given the convexity of $(x,y) \mapsto L_\rho^s(x,y,\z^{k-1},\omega_\s^{k})$ and the convexity of $\conv(\V^t_s)$,
the necessary and sufficient condition
for optimality
\begin{equation}\label{LemmaSDMEq1}
 \nabla_{(x,y)} L_\rho^s(x_s^t,y_s^t,\z^{k-1},\omega_\s^{k}) \left[ \begin{array}{c} x-x_s^t \\ y - y_s^t \end{array} \right] \ge 0 \quad \text{for all} \; (x,y) \in \conv(\V^t_s)
\end{equation}
is satisfied.
By assumption, condition~\eqref{LemmaSDMEq0} is satisfied, $\conv(\K_\s)$ is likewise convex, and so the resulting \emph{non}-satisfaction of the necessary and sufficient condition of optimality for the problem in~\eqref{LemmaSDMEq0} takes the form
\begin{equation}\label{LemmaSDMEq2}
\min_{x,y} \braces{\nabla_{(x,y)} L_\rho^s(x_s^t,y_s^t,\z^{k-1},\omega_\s^{k}) \left[ \begin{array}{c} x-x_s^t \\ y - y_s^t \end{array} \right] \st (x,y) \in \conv(\K_s) } < 0.
\end{equation}
In fact, during SDM iteration $t+1$, an optimal solution $(\widehat{x}_s,\widehat{y}_s)$ to the problem in condition~\eqref{LemmaSDMEq2} is computed in Line~\ref{DirFindingSDM} of Algorithm~\ref{AlgSDM}.
Therefore, by the satisfaction of condition~\eqref{LemmaSDMEq1} and the optimality of $(\widehat{x}_s,\widehat{y}_s)$ for the problem of condition~\eqref{LemmaSDMEq2}, which is also satisfied, we have $(\widehat{x}_s,\widehat{y}_s) \not\in \conv(\V^t_s)$. By construction, $\V^{t+1}_s \gets \V^t_s \cup \braces{(\widehat{x}_s,\widehat{y}_s)}$,
so that $\conv(\V^{t+1}_s) \supset \conv(\V^{t}_s)$ must hold.
\end{proof}

The FW-PH algorithm is stated in pseudocode-form in
Algorithm~\ref{AlgPHFW}.  Similar to Algorithm~\ref{AlgPH}, the
parameter $\epsilon$ is a convergence tolerance, and $k_{max}$ is the
maximum number of (outer) iterations.  The parameter $t_{max}$ is the
maximum number of (inner) SDM iterations in Algorithm~\ref{AlgSDM}.

The parameter $\alpha \in \reals$ affects the initial linearization
point $\tilde{x}_s$ of the SDM method.  Any value $\alpha \in \reals$
may be used, but the use of $\tilde{x}_s = (1 - \alpha) z^{k-1} +
\alpha x_s^{k-1}$ in Line \ref{InitialX} is a crucial component in the
efficiency of the FW-PH algorithm, as it enables the computation of
a valid dual bound, $\phi^k$, at \emph{each} iteration of FW-PH without
the need for additional MILP subproblem solutions.  Specifically, we have the following result.

\begin{proposition}\label{PropFreeLB}
Assume that the precondition $\sum_{\s\in\SC}p_\s\omega^0_\s=0$ holds for Algorithm~\ref{AlgPHFW}.
At each iteration $k\ge 1$ of Algorithm~\ref{AlgPHFW}, the value, $\phi^k$, calculated at Line \ref{ComputePhi0}, is the value of the Lagrangian relaxation $\phi(\cdot)$ evaluated at a Lagrangian dual feasible point, and hence provides a finite lower bound on $\optval^{LD}$.
\end{proposition}
\begin{proof}
Since $\sum_{\s\in\SC}p_\s\omega^0_\s=0$ holds and, by construction, $0 = \sum_{\s \in \SC} p_\s (x_\s^0 - z^0)$, we have $\sum_{\s\in\SC}p_\s\omega^1_\s=0$ also. We proceed by induction on $k \ge 1$.
At iteration $k$, the problem solved for each $\s\in\SC$ at Line \ref{DirFindingSDM} in the first iteration ($t=1$) of Algorithm \ref{AlgSDM} may be solved with
the same optimal value by exchanging $\mathcal{V}(\conv(\K_\s))$ for $\K_\s$; this follows from the linearity of the objective function.
Thus, an optimal solution computed at Line \ref{DirFindingSDM} may be used in the computation of $\phi_\s(\widetilde{\omega}_\s^k)$ carried out in Line~\ref{SDMEvalPhiS}, 
where
\begin{eqnarray*}
  \widetilde{\omega}_\s^{k} & := & \widehat{\omega}_s^1 = \omega_\s^{k} +
  \rho(\widetilde{x}_\s - z^{k-1})
  = \omega_\s^{k} + \rho((1-\alpha)z^{k-1} + \alpha x_\s^{k-1} -
  z^{k-1})\\
  & = & \omega_\s^k + \alpha \rho(x_\s^{k-1} - z^{k-1}).
\end{eqnarray*}
By construction, we have at each iteration $k \ge 1$ in Algorithm~\ref{AlgPHFW} that
$$\sum_{\s \in \SC} p_s (x_\s^{k-1}- z^{k-1}) = 0 \quad \text{and}\quad \sum_{\s \in
  \SC} p_s \omega_s^k = 0,$$
which establishes that $\sum_{\s \in \SC}
p_s \widetilde{\omega}_\s^{k} = 0$. Thus, $\widetilde{\omega}^k$ is feasible for the Lagrangian dual problem, so that $\phi(\widetilde{\omega}^k)=\sum_{\s \in \SC} p_\s \phi_\s^k$, and, since each $\phi_s^k$ is the optimal value of a bounded and feasible mixed-integer linear program,
we have $-\infty < \phi(\widetilde{\omega}^k) < \infty$.
\end{proof}

\begin{algorithm}[t]
\caption{FW-PH applied to problem~\eqref{EqSMIPEFSConv}. \label{AlgPHFW}}
\begin{algorithmic}[1]
\Function{FW-PH}{$(\V_\s^0)_{\s \in \SC}$, $(x_\s^0,y_\s^0)_{\s \in \SC}$, $\omega^0$, $\rho$, $\alpha$, $\epsilon$, $k_{max}$, $t_{max}$}
    \State $\z^0 \gets  \sum_{\s \in \SC} p_\s x_\s^0$
    \State $\omega_\s^1 \gets \omega_\s^0 + \rho ({x}_\s^0 - \z^0)$, for $\s \in \SC$
    \For{$k=1,\dots,k_{max}$}
    		
	 	
	 	\For{$\s \in \SC$}
		\State $\widetilde{x}_s \gets (1-\alpha) z^{k-1} +
                \alpha x_s^{k-1}$ \label{InitialX}
	 	\State $[x_\s^k, y_\s^k, \V_s^k, \phi_\s^k]  \gets {\rm SDM}(\V_s^{k-1}, \widetilde{x}_s, \omega_s^k, z^{k-1}, t_{max},0)$ \label{SDMCall}
	 	\EndFor
	 \State $\phi^k \gets \sum_{\s \in \SC} p_\s\phi_\s^k$ \label{ComputePhi0}
    	\State $\z^k \gets \sum_{\s \in \SC} p_\s x_\s^k$
    	\If{$\sqrt{\sum_{\s \in \SC} p_\s \norm{x_\s^k-\z^{k-1}}_2^2} < \epsilon $}
		\State {\bf return} { $((x_\s^{k},y_\s^{k})_{\s \in \SC},\z^{k}, \omega^{k}, \phi^{k})$}
	\EndIf
	\State $\omega_\s^{k+1} \gets \omega_\s^{k} + \rho(x_\s^{k}-\z^{k})$, for $\s \in \SC$
    \EndFor
    \State {\bf return} { $\left((x_s^{k_{max}},y_s^{k_{max}})_{\s \in \SC},\z^{k_{max}}), \omega^{k_{max}}, \phi^{k_{max}}\right)$}
\EndFunction
\end{algorithmic}
\end{algorithm}


We establish convergence \red{of} Algorithm~\ref{AlgPHFW} for any $\alpha \in \reals$ and $t_{max} \ge 1$.
For the special case where we perform only one iteration of SDM for
each outer iteration ($t_{max}=1$), we require the additional
assumption that \blue{the initial scenario vertex sets share a common point. More precisely, we require the assumption}
\begin{equation}
\bigcap_{\s \in \SC} \Proj_x (\conv(\V_s^0)) \neq \emptyset \label{FeasCond}
\end{equation}
which can, in practice, be effectively handled through appropriate initialization, under the standard assumption of relatively complete recourse. \blue{We describe one initialization procedure in Section \ref{Sec4}}.

\begin{proposition}\label{PropFWPH-SMIP}
Let the convexified separable deterministic equivalent SMIP~\eqref{EqSMIPEFSConv} have an optimal solution,
and let Algorithm~\ref{AlgPHFW} be applied to~\eqref{EqSMIPEFSConv} with $k_{max}=\infty$, $\epsilon = 0$, $\alpha \in
\reals$, and $t_{max} \geq 1$.
If either $t_{max} \ge 2$ or \eqref{FeasCond} holds, then $\lim_{k \to \infty} \phi^k = \optval^{LD}$.
\end{proposition}
\begin{proof}
First note that for any  $t_{max} \ge 1$, the sequence of inner approximations $\conv(\V_\s^k) $, $\s \in \SC$, will stabilize, in that, for some threshold $0 \le \bar{k}_\s$, we have for all $k \ge \bar{k}_\s$
\begin{equation}\label{Stabilisation}
\conv(\V_\s^k) =: \overline{\F}_\s \subseteq \conv(\K_\s).
\end{equation}
This follows due to the assumption that each expansion of the inner approximations $\conv(\V_\s^k) $ take the form $\V_\s^{k} \gets \V_\s^{k-1} \cup \braces{(\widehat{x}_\s,\widehat{y}_\s)}$, where $(\widehat{x}_\s,\widehat{y}_\s)$ is a vertex of $\conv(\K_\s)$. Since each polyhedron $\conv(\K_\s)$, $\s \in \SC$ has only a finite number of such vertices,
the stabilization~\eqref{Stabilisation} must occur at some $\bar{k}_\s < \infty$.

For $t_{\max} \ge 2$, the stabilizations~\eqref{Stabilisation}, $\s \in \SC$, are reached at some iteration $\bar{k}:=\max_{\s \in \SC} \braces{\bar{k}_\s}$.
Noting that $\overline{\F}_\s =\conv(\V_\s^k)$ for $k > \overline{k}$ we must have
\begin{equation}\label{TildeOpt}
(x_\s^k,y_\s^k) \in \argmin_{x,y} \braces{ L_\s^{\rho}(x,y,\z^{k-1},\omega_\s^{k}) : \\ (x,y) \in \conv(\K_\s)}.
\end{equation}
Otherwise, due to Lemma~\ref{FiniteSDMTermination}, the call to SDM on Line~\ref{SDMCall} must
return $\V_\s^k \supset \V_\s^{k-1}$, contradicting  the finite stabilization~\eqref{Stabilisation}. Therefore, the $k \ge \bar{k}$ iterations of Algorithm~\ref{AlgPHFW} are identical to Algorithm~\ref{AlgPH} iterations, and so Proposition~\ref{PropPH} implies that
$\lim_{k \to \infty} x_\s^k - \z^k = 0$, $\s \in \SC$, and $\lim_{k \to \infty} \phi(\omega^k) = \optval^{LD}$.
By the continuity of $\omega \mapsto \phi_\s(\omega)$ for each $\s \in \SC$, we have
$\lim_{k \to \infty} \phi^k = \lim_{k \to \infty} \sum_{\s \in \SC} p_\s \phi_\s(\omega_\s^k+\alpha(x_\s^{k-1}-z^{k-1})) = \lim_{k \to \infty} \sum_{\s \in \SC} p_\s \phi_\s(\omega_\s^k) =\lim_{k \to \infty} \phi(\omega^k) = \optval^{LD}$ for all $\alpha \in \reals$.

In the case $t_{max}=1$, we have at each iteration $k \ge 1$ the optimality
$$
(x_\s^k,y_\s^k) \in \argmin_{x,y} \braces{ L_\s^\rho(x_\s,y_\s,\z^{k-1},\omega_\s^k)  \st   (x_\s,y_\s) \in \conv(\V_\s^k) }.
$$
By the stabilization~\eqref{Stabilisation}, the iterations $k \ge \bar{k}$ of Algorithm~\ref{AlgPHFW} are identical to PH iterations applied to the restricted problem
\begin{equation}\label{EqSMIPEFSConvRestr}
\min_{x,y,\z} \braces{\sum_{\s \in \SC} p_\s(c^\top
  x_\s + q_\s^\top y_\s) \st (x_\s,y_\s) \in \overline{\F}_\s,
  \ \forall \s \in \SC, x_\s = \z, \ \forall \s \in \SC}.
\end{equation}
We have initialized the sets $(\V_\s^0)_{\s \in \SC}$ such that $\cap_{\s \in \SC} \Proj_x \conv(\V_\s^0)\neq \emptyset$, so since the inner approximations to $\conv(K_s)$
only expand in the algorithm,
$\cap_{\s \in \SC} \Proj_x (\overline{\F}_s) \neq \emptyset$.
Therefore, problem~\eqref{EqSMIPEFSConvRestr} is a feasible and bounded linear program, and so
the PH convergence described in Proposition~\ref{PropPH} with $D_\s = \overline{\F}_\s$, $\s \in \SC$, holds for its application to problem~\eqref{EqSMIPEFSConvRestr}.
That is, for each $\s \in \SC$, we have 1) $\lim_{k \to \infty} \omega_\s^k = \omega_\s^*$ and $\lim_{k \to \infty}
(x_s^k - \z^k) = 0$; and 2) for all limit points $((x_\s^*,y_\s^*)_{\s \in \SC}, \z^*)$, we have the feasibility and
optimality of the limit points, which implies $x_\s^* =
\z^*$ and
\exclude{
\begin{align}
0=\min_{x,y} \braces{\left[ \begin{array}{c} c+\omega_\s^*  \\ q_\s \end{array}\right]^\top \left[ \begin{array}{c} x-x^* \\ y-y^* \end{array} \right] : (x,y) \in \overline{\F}_\s}. \label{MinEqZero}
\end{align}
}
\begin{equation}
  \min_{x,y} \braces{ (c+\omega_\s^*)^\top(x-x^*) + q_\s^\top
    (y-y^*) \st (x,y) \in  \overline{\F}_\s} = 0
  \label{MinEqZero}
\end{equation}
Next, for each $\s \in \SC$ the compactness of $\conv(\K_\s) \supseteq
\overline{\F}_\s$, the continuity of the minimum value function
\exclude{
\[
\omega \mapsto \min_{x,y} \braces{\left[ \begin{array}{c} c+\omega  \\ q_\s \end{array}\right]^\top \left[ \begin{array}{c} x \\ y \end{array} \right] : (x,y) \in \overline{\F}_\s},
\]
}
\[ \omega \mapsto \min_{x,y} \braces{ (c+\omega)^\top x + q_\s^\top y
  \st (x,y) \in \overline{\F}_\s  }, \]
and the limit $\lim_{k \to \infty} \widetilde{\omega}_\s^{k+1} = \lim_{k \to \infty} {\omega}_\s^{k+1}+\alpha\rho(x_\s^k
- \z^k) = \omega_\s^*$, together imply that
\exclude{
\begin{equation}\label{LimZero}
\lim_{k \to \infty} \min_{x,y} \braces{\left[ \begin{array}{c} c+\widetilde{\omega}_\s^{k+1}  \\ q_\s \end{array}\right]^\top \left[ \begin{array}{c} x-x^k \\ y-y^k \end{array} \right] : (x,y) \in \overline{\F}_\s} = 0.
\end{equation}
}
\begin{equation}\label{LimZero}
  \lim_{k \to \infty} \min_{x,y} \braces{
    (c+\widetilde{\omega}_\s^{k+1})^\top (x-x^k) +  q_\s^\top (y-y^k)
    \st (x,y) \in \overline{\F}_\s} = 0.
\end{equation}

Recall that $\widetilde{\omega}_s^k = \omega_s^k + \rho\alpha(x_s^{k-1} - z^{k-1})$ is the $t=1$ value of
$\widehat{\omega}_s^t$ defined in Line \ref{SDMomega} of Algorithm \ref{AlgSDM}.
Thus, for $k+1 > \bar{k}$, we have due to the
stabilization~\eqref{Stabilisation} that
\exclude{
\begin{align}
&\min_{x,y} \braces{\left[ \begin{array}{c} c+\widetilde{\omega}_\s^{k+1}  \\ q_\s \end{array}\right]^\top \left[ \begin{array}{c} x-x^k \\ y-y^k \end{array} \right] : (x,y) \in \overline{\F}_\s} \nonumber\\
&= \min_{x,y} \braces{\left[ \begin{array}{c} c+\widetilde{\omega}_\s^{k+1}  \\ q_\s \end{array}\right]^\top \left[ \begin{array}{c} x-x^k \\ y-y^k \end{array} \right] \st (x,y) \in \conv(\K_\s)} \label{Equality}
\end{align}
}
\begin{multline}
   \min_{x,y} \braces{(c+\widetilde{\omega}_\s^{k+1})^\top (x-x^k) +  q_\s^\top (y-y^k)
     \st (x,y) \in \overline{\F}_\s} = \\
   \min_{x,y} \braces{(c+\widetilde{\omega}_\s^{k+1})^\top (x-x^k) +  q_\s^\top (y-y^k)
  \st (x,y) \in \conv(\K_\s)}
  \label{Equality}
\end{multline}
If equality~\eqref{Equality} does not hold, then the inner approximation expansion \linebreak $\overline{\F}_\s \subset \conv(\V_\s^{k+1})$ must occur, 
since a point $(\widehat{x}_\s,\widehat{y}_\s) \in \conv(\K_\s)$ that can be strictly separated from $\overline{\F}_\s$ would have been discovered during the iteration $k+1$ execution of Algorithm \ref{AlgSDM}, Line~\ref{DirFindingSDM}, $t=1$.  The expansion $\overline{\F}_\s \subset \conv(\V_\s^{k+1})$ contradicts the finite stabilization~\eqref{Stabilisation}, and so~\eqref{Equality} holds.
Therefore, the
equalities~\eqref{LimZero} and~\eqref{Equality} imply that
\exclude{
\begin{equation}
\lim_{k \to \infty} \min_{x,y} \braces{\left[ \begin{array}{c}
      c+\widetilde{\omega}_\s^{k+1}  \\ q_\s \end{array}\right]^\top
  \left[ \begin{array}{c} x-x^k \\ y-y^k \end{array} \right] \st (x,y) \in \conv(\K_\s)} = 0.
\end{equation}
}
\begin{equation}
  \lim_{k \to \infty} \min_{x,y} \braces{
    (c+\widetilde{\omega}_\s^{k+1})^\top(x-x^k) + q_\s^\top (y-y^k)
    \st  (x,y) \in \conv(\K_\s)} = 0.
\end{equation}

Our argument has shown that for all limit points $(x_\s^*,y_\s^*)$,
$\s \in \SC$, the following stationarity condition is satsfied:
\exclude{
\begin{equation}
\left[ \begin{array}{c} c+\omega_\s^*  \\ q_\s \end{array}\right]^\top \left[ \begin{array}{c} x-x_\s^* \\ y-y_\s^* \end{array} \right] \ge 0 \;\; \forall \;(x,y) \in \conv(\K_\s), \label{StationarityK}
\end{equation}
}
\begin{equation}
  (c+\omega_\s^*)^\top(x-x_\s^*) + q_\s^\top(y-y_\s^*) \geq 0 \quad
  \forall (x,y) \in \conv(\K_\s), \label{StationarityK}
\end{equation}
which together with the feasibility $x_\s^* = \z^*$, $\s \in \SC$ implies that each limit point $((x_\s^*,y_\s^*)_{\s \in \SC}, \z^*)$ is optimal for problem~\eqref{EqSMIPEFSConv} and $\omega^*$ 
is optimal for the dual problem~\eqref{eq:zldprob}.

Thus, for all $t_{max} \ge 1$, we have shown $\lim_{k \to \infty} (x_\s^k - \z^k) = 0$, $\s \in \SC$, and $\lim_{k \to \infty} \phi(\omega^k) = \optval^{LD}$.
By similar reasoning used in the $t_{max} \ge 2$ case, it is straightforward that for all $\alpha \in \reals$, we also have $\lim_{k \to \infty} \phi^k = \optval^{LD}$.
\end{proof}

While using a large value of $t_{max}$ more closely matches Algorithm~\ref{AlgPHFW} to the original PH algorithm as
described in Algorithm \ref{AlgPH},
we are motivated to use a small value of $t_{max}$ since the work per iteration is proportional to $t_{max}$.
Specifically, each iteration requires solving $t_{max}|\SC|$ MILP subproblems, and $t_{max}|\SC|$ continuous \red{convex} quadratic subproblems.
(For reference, Algorithm~\ref{AlgPH} applied to problem~\eqref{EqSMIPEFS} requires the solution of $|\SC|$
MIQP subproblems for each $\omega$ update and $|\SC|$ MILP subproblems for each Lagrangian bound $\phi$ computation.)

\section{Numerical Experiments} \label{Sec4}

We performed computations using a C++ implementation of Algorithm~\ref{AlgPH} ($\F_\s = \K_\s$, $\s \in \SC$) and \red{Algorithm}~\ref{AlgPHFW} using CPLEX 12.5~\cite{CPLEX12-5} as the solver \blue{for all subproblems}. For reading SMPS files into scenario-specific subproblems and for their interface with CPLEX, we used modified versions of the COIN-OR~\cite{COIN-ORURL} Smi and Osi libraries.
The computing environment is the Raijin cluster maintained by Australia's National Computing Infrastructure (NCI) and supported by the Australian Government~\cite{NCIURL}.
The Raijin cluster is a high performance computing (HPC) environment which has 3592 nodes (system units), 57472 cores of Intel Xeon E5-2670 processors with up to 8 GB PC1600 memory per core (128 GB per node). All experiments were performed in a serial setting using a single node and one thread per CPLEX solve.

In the experiments with Algorithms~\ref{AlgPH} and~\ref{AlgPHFW}, we set the convergence tolerance at \green{$\epsilon = 10^{-3}$}. For Algorithm~\ref{AlgPHFW}, we set $t_{max}=1$. \red{Also, for all experiments performed, we set $\omega^0 = 0$.}
In this case, convergence of our algorithm requires that \eqref{FeasCond} holds, which can be guaranteed during the initialization of the inner approximations $(\V^0_\s)_{\s \in \SC}$.
Under the standard assumption of {\it relatively complete resource}, i.e., the assumption that for all $x \in X$ and $s \in \SC$ there exists $y_\s$ such that $(x,y_\s) \in K_\s$, a straightforward mechanism for ensuring this assumption is to solve the recourse problems for any fixed $\widehat{x} \in X$.  Specifically, for each $\s \in \SC$, let \[ \widehat{y}_s \in \arg \min_y \{q_s^\top y \st (\widehat{x},y) \in K_s\}, \]
and initialize $\V_\s^0$ for each $\s \in \SC$ so that  $ \{(\widehat{x}, \widehat{y}_\s)\} \in \V_\s^0$.
Observe also that this initialization corresponds to a technique for computing a
feasible solution to the original problem~\eqref{EqSMIPEF}, which is independently useful for obtaining an {\it upper bound} on
$\optval^{SMIP}$.

For the computational experiments, we run the following initialization to obtain $(\V_\s^0)_{\s \in \SC}$ and $(x_\s^0,y_\s^0)_{\s \in \SC}$
that are input into Algorithm~\ref{AlgPHFW}:

\begin{algorithm}
\caption{Initialization step for FW-PH}
\begin{algorithmic}[1]
\State Precondition: Problem~\eqref{EqSMIPEF} has relatively complete recourse
\Function{FW-PH-Initialization}{$\omega^0$}
\For{$\s \in \SC$} \label{FWPH0InitBegin}
    \State $({x}_\s^0,{y}_s^0) \gets \argmin_{x,y}
    \braces{(c+\omega_\s^0)^\top x + q_\s^\top y \st (x,y) \in \red{\K_\s}}$
    \State $\V_\s^0 \gets \braces{({x}_\s^0,{y}_s^0)}$
    \If{$\s \neq 1$}
    	\State $\overline{y}_s \gets \argmin_{y}
    \braces{ q_\s^\top y \st (x_1^0,y) \in \red{\K_\s}}$
    \State $\V_\s^0 \gets \V_\s^0 \cup \braces{({x}_1^0,\overline{y}_s)}$
    \EndIf
  \EndFor  \label{FWPH0InitEnd}
  \State {\bf return} $\left(\V_\s^0, (x_\s^0,y_\s^0)\right)_{\s \in \SC}$
  \EndFunction
\end{algorithmic}
\end{algorithm}
If problem~\eqref{EqSMIPEF} does not have relatively complete recourse, then any means to compute a feasible solution to~\eqref{EqSMIPEF} may be employed to
initialize each $\V_\s^0$, $\s \in \SC$, \blue{in a way to satisfy \eqref{FeasCond}}.


Two sets of Algorithm~\ref{AlgPHFW} experiments correspond to variants considering $\alpha=0$ and $\alpha=1$. Computations were performed on four problems: the CAP instance 101 with the first 250 scenarios (CAP-101-250)~\cite{Bodur2014EtAl}, the DCAP instance DCAP233\_500 with 500 scenarios, the SSLP instances SSLP5.25.50 with 50 scenarios (SSLP-5-25-50) and  SSLP10.50.100 with 100 scenarios (SSLP-10-50-100). The latter three problems are described in detail in~\cite{NtaimoPhD2004,GTURL} and accessible at~\cite{GTURL}. For each problem, computations were performed for different penalty values $\rho>0$. The penalty values used in the experiments for the SSLP-5-25-50 instance were chosen to include those penalties that are tested in a computational experiment with PH whose results are depicted in Figure 2 of~\cite{GadeEtAl2016}. For the other problem instances, the set of penalty values $\rho$ tested is chosen to capture a reasonably wide range of performance potential for both PH and FW-PH.
All computational experiments were allowed to run for a maximum of two hours in wall \red{clock} time.

\addtolength{\tabcolsep}{-2pt}
\begin{table}[h]
  \centering
    \begin{tabular}{rccccccccc}
    \toprule
          & \multicolumn{3}{c}{Percentage gap} & \multicolumn{3}{c}{\# Iterations} & \multicolumn{3}{c}{Termination} \\
    \midrule
    \multirow{2}[2]{*}{Penalty} & \multicolumn{1}{c}{\multirow{2}[2]{*}{PH}} & \multicolumn{2}{c}{FW-PH} & \multicolumn{1}{c}{\multirow{2}[2]{*}{PH}} & \multicolumn{2}{c}{FW-PH} & \multirow{2}[2]{*}{PH} & \multicolumn{2}{c}{FW-PH} \\
          & \multicolumn{1}{c}{} & $\alpha=0$ & $\alpha=1$ & \multicolumn{1}{c}{} & $\alpha=0$ & $\alpha=1$ &       & $\alpha=0$ & $\alpha=1$  \\ \toprule
    20    & 0.08\% & 0.10\% & 0.11\% & 466   & 439   & 430   & T     & T     & T \\
    100   & 0.01\% & 0.00\% & 0.00\% & 178   & 406   & 437   & C     & T     & T \\
    500   & 0.07\% & 0.00\% & 0.00\% & 468   & 92    & 93    & T     & C     & C \\
    1000  & 0.15\% & 0.00\% & 0.00\% & 516   & 127   & 130   & T     & C     & C \\
    2500  & 0.34\% & 0.00\% & 0.00\% & 469   & 259   & 274   & T     & C     & C \\
    5000  & 0.66\% & 0.00\% & 0.00\% & 33    & 431   & 464   & C     & T     & T \\
    7500  & 0.99\% & 0.00\% & 0.00\% & 28    & 18    & 19    & C     & C     & C \\
    15000 & 1.59\% & 0.00\% & 0.00\% & 567   & 28    & 33    & T     & C     & C \\
    \bottomrule
    \end{tabular}%
  \caption{Result summary for CAP-101-250, with the absolute percentage gap based on the known optimal value 733827.3\label{tab:CAPResultsSum}}
\end{table}%

\begin{table}[h]
  \centering
    \begin{tabular}{rccccccccc}
    \toprule
          & \multicolumn{3}{c}{Percentage gap} & \multicolumn{3}{c}{\# Iterations} & \multicolumn{3}{c}{Termination} \\
    \midrule
    \multirow{2}[2]{*}{Penalty} & \multicolumn{1}{c}{\multirow{2}[2]{*}{PH}} & \multicolumn{2}{c}{FW-PH} & \multicolumn{1}{c}{\multirow{2}[2]{*}{PH}} & \multicolumn{2}{c}{FW-PH} & \multirow{2}[2]{*}{PH} & \multicolumn{2}{c}{FW-PH} \\
          & \multicolumn{1}{c}{} & $\alpha=0$ & $\alpha=1$ & \multicolumn{1}{c}{} & $\alpha=0$ & $\alpha=1$ &       & $\alpha=0$ & $\alpha=1$  \\ \toprule
    2     & 0.13\% & 0.12\% & 0.12\% & 1717  & 574   & 600   & T     & T     & T \\
    5     & 0.22\% & 0.09\% & 0.09\% & 2074  & 589   & 574   & T     & T     & T \\
    10    & 0.23\% & 0.07\% & 0.07\% & 2598  & 592   & 587   & T     & T     & T \\
    20    & 0.35\% & 0.07\% & 0.07\% & 1942  & 590   & 599   & T     & T     & T \\
    50    & 1.25\% & 0.06\% & 0.06\% & 2718  & 597   & 533   & T     & T     & T \\
    100   & 1.29\% & 0.06\% & 0.06\% & 2772  & 428   & 438   & T     & C     & C \\
    200   & 2.58\% & 0.06\% & 0.06\% & 2695  & 256   & 262   & T     & C     & C \\
    500   & 2.58\% & 0.07\% & 0.07\% & 2871  & 244   & 246   & T     & C     & C \\
    \bottomrule
    \end{tabular}%
  \caption{Result summary for DCAP-233-500, with the absolute percentage gap based on the best known lower bound 1737.7.  \label{tab:DCAPResultsSum}}
\end{table}%

\begin{table}[h]
  \centering
    \begin{tabular}{rccccccccc}
    \toprule
          & \multicolumn{3}{c}{Percentage gap} & \multicolumn{3}{c}{\# Iterations} & \multicolumn{3}{c}{Termination} \\
    \midrule
    \multirow{2}[2]{*}{Penalty} & \multicolumn{1}{c}{\multirow{2}[2]{*}{PH}} & \multicolumn{2}{c}{FW-PH} & \multicolumn{1}{c}{\multirow{2}[2]{*}{PH}} & \multicolumn{2}{c}{FW-PH} & \multirow{2}[2]{*}{PH} & \multicolumn{2}{c}{FW-PH} \\
          & \multicolumn{1}{c}{} & $\alpha=0$ & $\alpha=1$ & \multicolumn{1}{c}{} & $\alpha=0$ & $\alpha=1$ &       & $\alpha=0$ & $\alpha=1$  \\ \toprule
    1     & 0.30\% & 0.00\% & 0.00\% & 105   & 115   & 116   & C     & C     & C \\
    2     & 0.73\% & 0.00\% & 0.00\% & 51    & 56    & 56    & C     & C     & C \\
    5     & 0.91\% & 0.00\% & 0.00\% & 25    & 26    & 27    & C     & C     & C \\
    15    & 3.15\% & 0.00\% & 0.00\% & 12    & 16    & 17    & C     & C     & C \\
    30    & 6.45\% & 0.00\% & 0.00\% & 12    & 18    & 18    & C     & C     & C \\
    50    & 9.48\% & 0.00\% & 0.00\% & 18    & 25    & 26    & C     & C     & C \\
    100   & 9.48\% & 0.00\% & 0.00\% & 8     & 45    & 45    & C     & C     & C \\
    \bottomrule
    \end{tabular}%
  \caption{Result summary for SSLP-5-25-50, with the absolute percentage gap based on the known optimal value -121.6 \label{tab:SSLPResultsSum5-25-50}}
\end{table}%

\begin{table}[h]
  \centering
    \begin{tabular}{rccccccccc}
    \toprule
          & \multicolumn{3}{c}{Percentage gap} & \multicolumn{3}{c}{\# Iterations} & \multicolumn{3}{c}{Termination} \\
    \midrule
    \multirow{2}[2]{*}{Penalty} & \multicolumn{1}{c}{\multirow{2}[2]{*}{PH}} & \multicolumn{2}{c}{FW-PH} & \multicolumn{1}{c}{\multirow{2}[2]{*}{PH}} & \multicolumn{2}{c}{FW-PH} & \multirow{2}[2]{*}{PH} & \multicolumn{2}{c}{FW-PH} \\
          & \multicolumn{1}{c}{} & $\alpha=0$ & $\alpha=1$ & \multicolumn{1}{c}{} & $\alpha=0$ & $\alpha=1$ &       & $\alpha=0$ & $\alpha=1$  \\ \toprule
    1     & 0.57\% & 0.22\% & 0.23\% & 126   & 234   & 233   & T     & T     & T \\
    2     & 0.63\% & 0.03\% & 0.03\% & 127   & 226   & 228   & T     & T     & T \\
    5     & 1.00\% & 0.00\% & 0.00\% & 104   & 219   & 220   & C     & T     & T \\
    15    & 2.92\% & 0.00\% & 0.00\% & 33    & 45    & 118   & C     & C     & C \\
    30    & 4.63\% & 0.00\% & 0.00\% & 18    & 21    & 22    & C     & C     & C \\
    50    & 4.63\% & 0.00\% & 0.00\% & 11    & 26    & 27    & C     & C     & C \\
    100   & 4.63\% & 0.00\% & 0.00\% & 9     & 43    & 45    & C     & C     & C \\
    \bottomrule
    \end{tabular}%
  \caption{Result summary for SSLP-10-50-100, with the absolute percentage gap based on the known optimal value -354.2 \label{tab:SSLPResultsSum}}
\end{table}%

Tables~\ref{tab:CAPResultsSum}--\ref{tab:SSLPResultsSum} provide a summary indicating the quality of the Lagrangian bounds $\phi$ computed at the end of each experiment for the four problems with varying penalty parameter $\rho$. In each of these tables, the first column lists the values of the penalty parameter $\rho$, while the following are presented for PH and FW-PH (for both $\alpha=0$ and $\alpha=1$) computations in the remaining columns: 1) the absolute percentage gap \blue{$\left|\frac{\optval^*-\phi}{\optval^*}\right|*100\%$} between the computed Lagrangian bound $\phi$ and some reference value $\optval^*$ that is either a known optimal value for the problem, or a known best \blue{upper} bound thereof (column ``Percentage Gap''); 2) the total number of dual updates (``\# Iterations''); and 3) the indication of whether the algorithm terminated due to the time limit, indicated by letter ``T'', or the satisfaction of the convergence criterion $\sqrt{\sum_{\s \in \SC} p_\s\norm{x_\s^k-\z^{k-1}}_2^2} < \epsilon $, indicated by letter ``C'' (column ``Termination'').

The following observations can be made from the results presented in Tables~\ref{tab:CAPResultsSum}--\ref{tab:SSLPResultsSum}. First, for small values of the penalty $\rho$, there is no clear preference between the bounds $\phi$ generated by PH and FW-PH. However, for higher penalties, the bounds $\phi$ obtained by FW-PH are consistently of better quality (i.e., higher) than those obtained by PH, regardless of the variant used (i.e. $\alpha=0$ or $\alpha=1$).
This tendency is typically illustrated, for example, in Table 2, where the absolute percentage gap of the Lagrangian lower bound with the known optimal value was 0.06\% with $\rho=200$ for FW-PH ($\alpha = 0$), while it was 2.58\% for the same value of $\rho$ for PH. This improvement is consistently observed for the other problems and the other values of $\rho$ that are not too close to zero.
Also, FW-PH did not terminate with suboptimal convergence or display cycling behavior for any of the penalty values $\rho$ in any of the problems considered. For example, all experiments considered in Table~\ref{tab:SSLPResultsSum5-25-50} terminated due to convergence. The percentage gaps suggest that the convergence for PH was suboptimal, while it was optimal for FW-PH.  Moreover, it is possible to see from these tables that the quality of the bounds $\phi$ obtained using FW-PH were not as sensitive to the value of the penalty parameter $\rho$ as obtained using PH.

The FW-PH with $\alpha=0$ versus PH convergence profiles for the experiments performed are given in Figures~\ref{FigCAP101}--\ref{FigSSLP10-50-100}, in which we provide plots of wall time versus Lagrangian bound values based on profiling of varying penalty. The times scales for the plots have been set such that trends are meaningfully depicted (1000s for CAP-101-250 and DCAP-233-500, 100 seconds for SSLP-5-25-50, and 3000s for SSLP-10-50-100). The trend of the Lagrangian bounds is depicted with solid lines for FW-PH with $\alpha = 0$ and with dashed lines for PH. Plots providing the same comparison for FW-PH with $\alpha = 1$ are provided in Appendix~\ref{ApxA}.

As seen in the plots of Figures~\ref{FigCAP101}--\ref{FigSSLP10-50-100}, the Lagrangian bounds $\phi$ generated with PH tend to converge suboptimally, often displaying cycling, for large penalty values. In terms of the quality of the bounds obtained, while there is no clear winner when low penalty $\rho$ values are used, for large penalties, the quality of the bounds $\phi$ generated with FW-PH is consistently better than for the bounds generated with PH, regardless of the $\alpha$ value. {This last observation is significant because the effective use of large penalty values $\rho$ in methods based on augmented Lagrangian relaxation tends to yield the most rapid early iteration improvement in the Lagrangian bound; this point is most clearly illustrated in the plot of Figure~\ref{FigSSLP5-25-50}.}

\begin{figure}[hbtp]
\includegraphics[trim = 30mm 10mm 25mm 10mm, clip, width=0.9\textwidth]{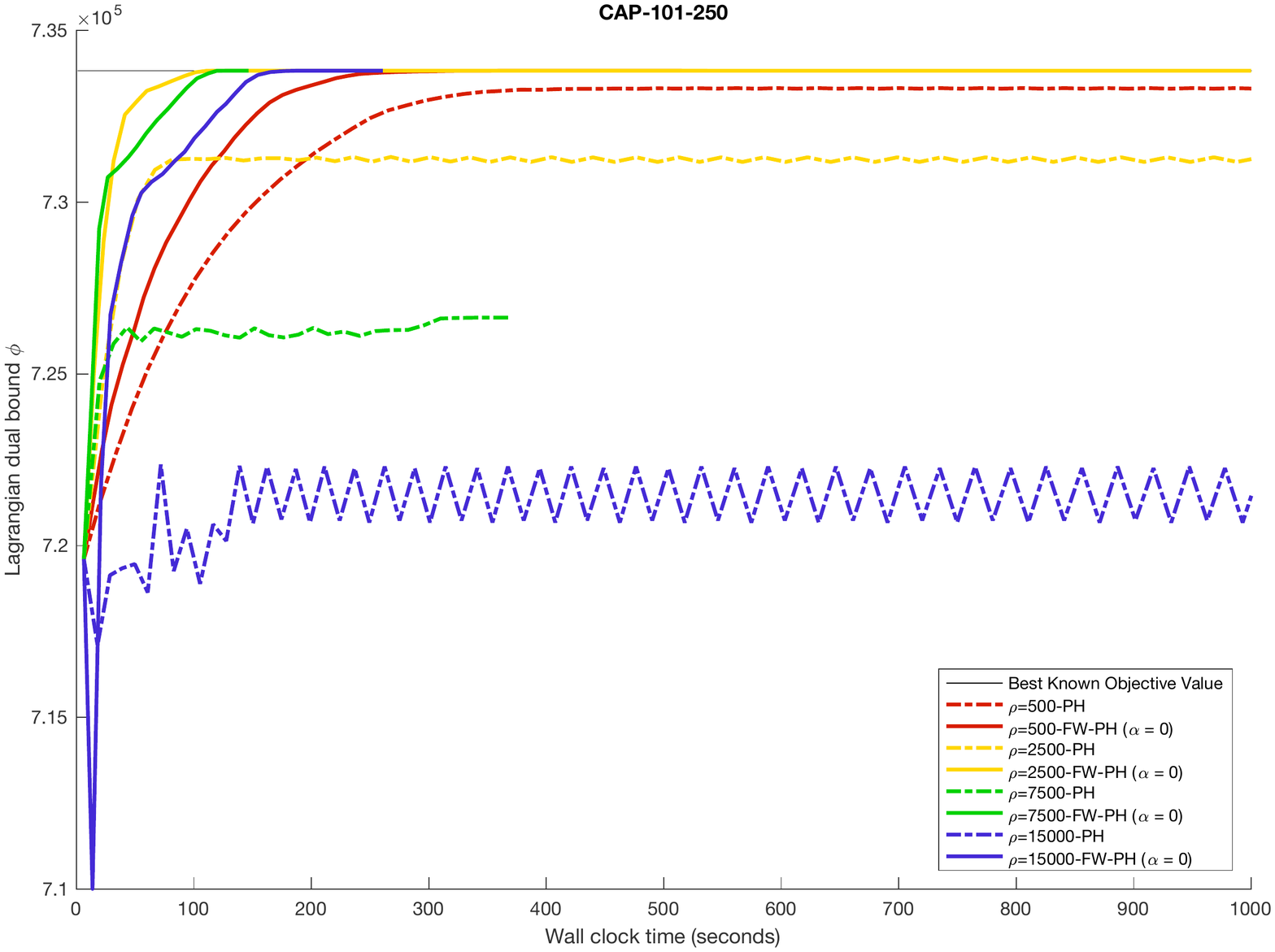}
\caption{Convergence profile for CAP-101-250 (PH and FW-PH with $\alpha = 0$)\label{FigCAP101}}
\end{figure}
\begin{figure}[hbtp]
\includegraphics[trim = 30mm 10mm 25mm 10mm, clip,width=0.9\textwidth]{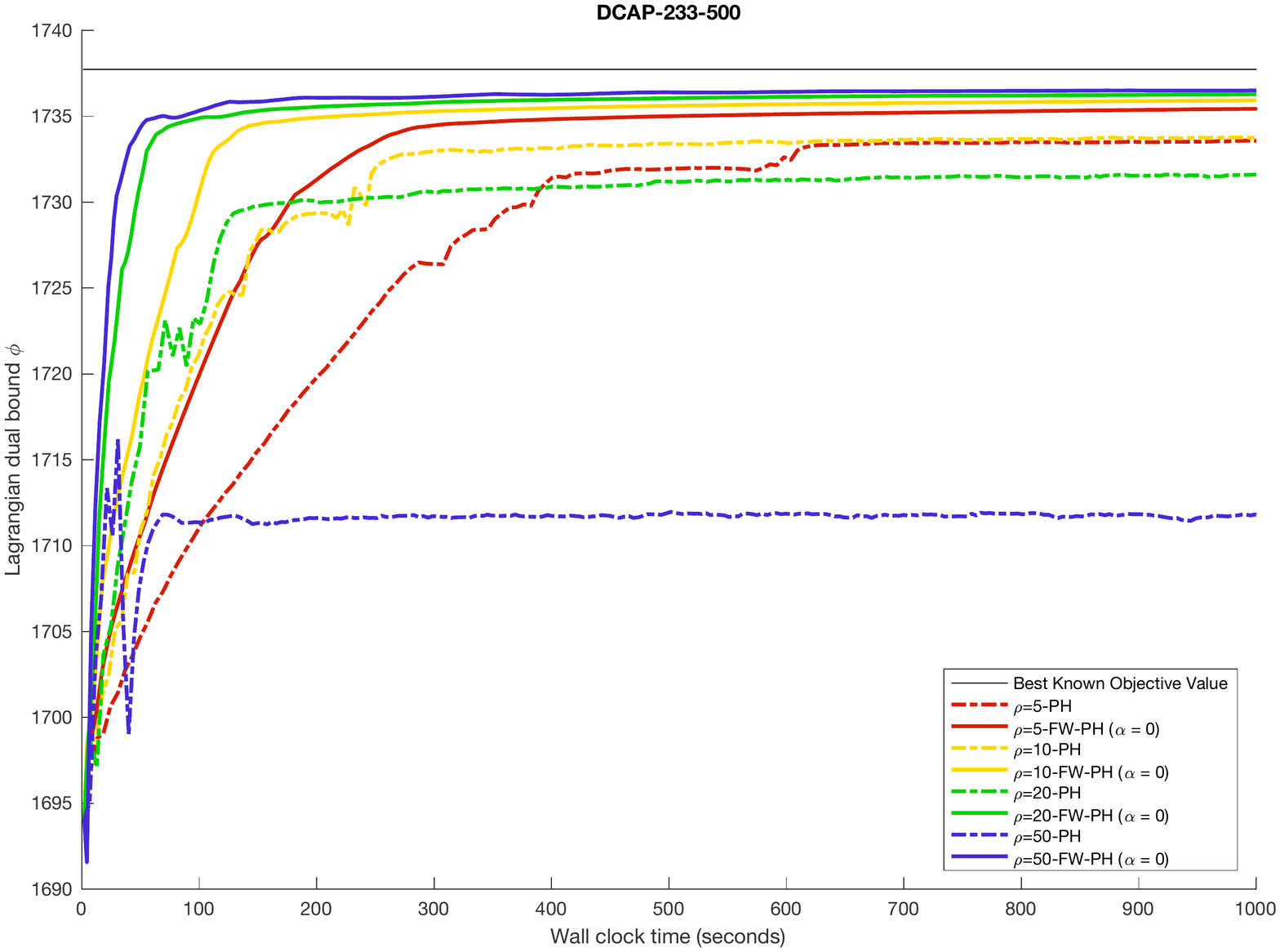}
\caption{Convergence profile for DCAP-233-500 (PH and FW-PH with $\alpha = 0$)\label{FigDCAP233500}}
\end{figure}
\begin{figure}[hbtp]
\includegraphics[trim = 30mm 10mm 25mm 10mm, clip,width=0.9\textwidth]{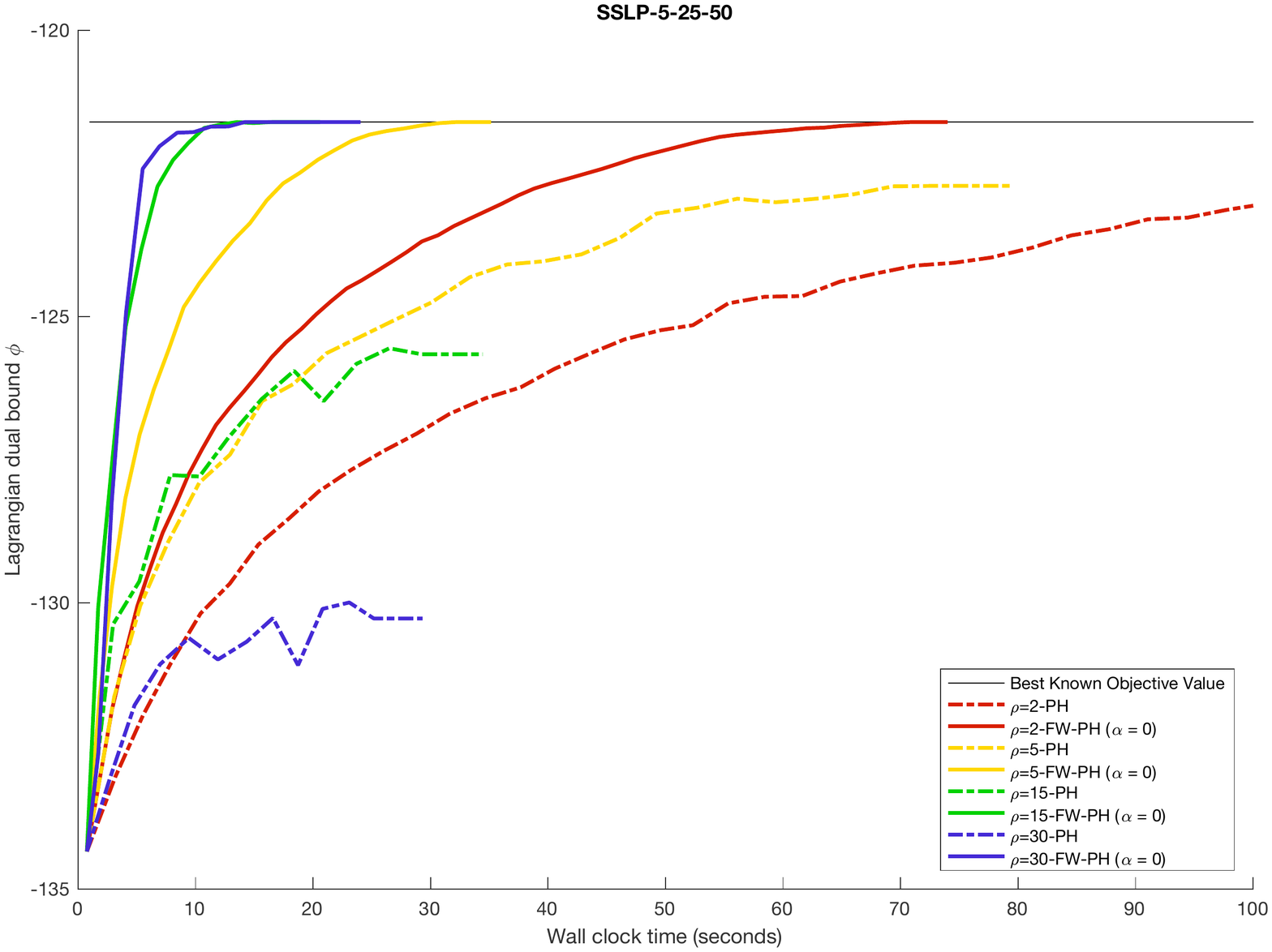}
\caption{Convergence profile for SSLP-5-25-50 (PH and FW-PH with $\alpha = 0$)\label{FigSSLP5-25-50}}
\end{figure}
\begin{figure}[hbtp]
\includegraphics[trim = 30mm 10mm 25mm 10mm, clip,width=0.9\textwidth]{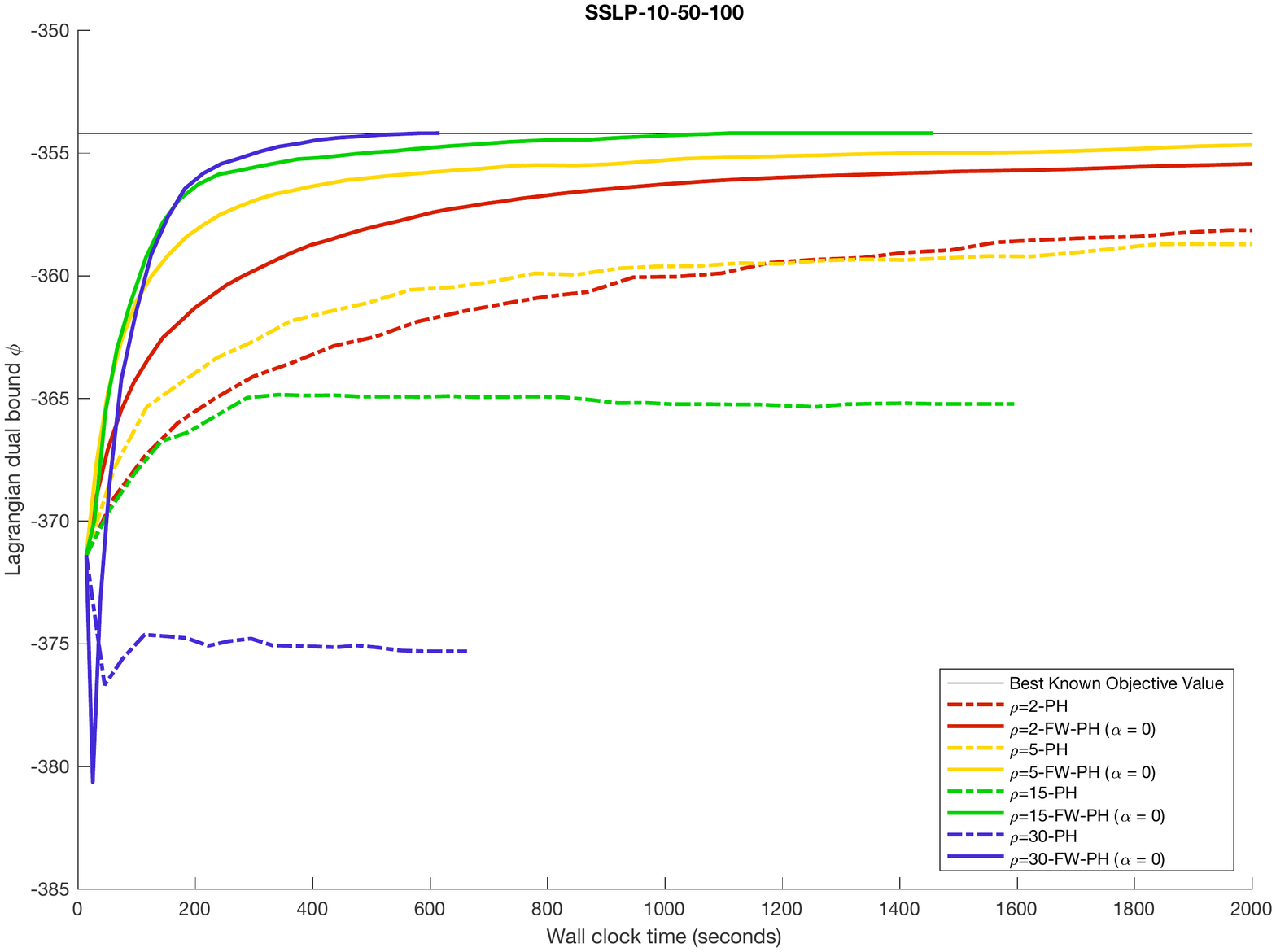}
\caption{Convergence profile for SSLP-10-50-100 (PH and FW-PH with $\alpha = 0$)\label{FigSSLP10-50-100}}
\end{figure}

\section{Conclusions and future work} \label{Sec5}

In this paper, we have presented an alternative approach to compute nonanticipativity Lagrangian
bounds associated with SMIPs that combines ideas from the Progressive Hedging (PH) and the Frank-Wolfe (FW) based methods. We first note that while Lagrangian bounds can be recovered with  PH, this requires---for each iteration and each scenario---the solution of an additional MILP subproblem in addition to the MIQP subproblem. Furthermore, \blue{when using the PH method directly,} the Lagrangian bounds \blue{may} converge suboptimally, cycle (for large penalties), or converge very slowly (for small penalties).

To overcome the lack of theoretical support for the above use of PH, we first described a straightforward integration of PH and a FW-like approach such as the Simplicial Decomposition Method (SDM), where SDM is used to compute the primal updates in PH. Its convergence only requires noting that SDM applied to a convex problem with a bounded polyhedral constraint set terminates finitely with optimal convergence. However, for the stated goal of computing high\blue{-}quality Lagrangian bounds efficiently, the benefits of relying on the optimal convergence of SDM is far outweighed by the computational costs incurred.

As an alternative, we \blue{propose} the contributed algorithm, FW-PH, that is analyzed under general assumptions on how the Lagrangian bounds are computed and on the number of SDM iterations used for each dual update. Furthermore, under mild assumptions on the initialization of the algorithm, FW-PH only requires the solution of a MILP subproblem and a continuous \red{convex} quadratic subproblem for each iteration and each scenario. FW-PH is versatile enough to handle a wide range of SMIPs with integrality restrictions in any stage, while providing rapid improvement in the Lagrangian bound in the early iterations that is consistent across a wide range of penalty parameter values. Although we have opted to focus on two-stage problems with recourse, the generalization of the proposed approach to the multi-stage case is also possible.

Numerical results are encouraging as they suggest that the proposed FW-PH method applied to SMIP problems usually outperforms the traditional PH method 
with respect to how quickly the quality of the generated Lagrangian bounds improves. \green{This is especially true with the use of larger penalty values.} For all problems considered and for all but the smallest penalties considered,  the FW-PH method displayed better performance over PH in terms of the quality of the final Lagrangian bounds at the end of the allotted \red{wall clock} time.

The improved performance of FW-PH over PH for large penalties is significant because it is the effective use of large penalties enabled by FW-PH that yields the most rapid initial dual improvement. This last feature of FW-PH would be most helpful in its use within a branch-and-bound or branch-and-cut framework for providing strong lower bounds (in the case of minimization). In addition to being another means to compute Lagrangian bounds, PH would still have a role in such frameworks as a heuristic for computing a primal feasible solution to the SMIP, thus providing (in the case of minimization) an upper bound on the optimal value.

Future research on this subject includes the following. First, FW-PH inherits the potential for parallelization from PH.
Experiments for exploring the benefit of parallelization are therefore warranted. Second, the theoretical support of FW-PH can be strengthened with a better understanding of the behavior of PH (and its generalization ADMM) applied to infeasible problems.
Finally, FW-PH can benefit from a better understanding of how the proximal term penalty coefficient can be varied to improve performance.




\bibliographystyle{plain}
\bibliography{fwph}

\pagebreak

\appendix

\section{Additional plots for PH vs. FW-PH for $\alpha = 1$} \label{ApxA}
\begin{figure}[h]
\includegraphics[trim = 30mm 10mm 30mm 10mm, clip, width=0.77\textwidth]{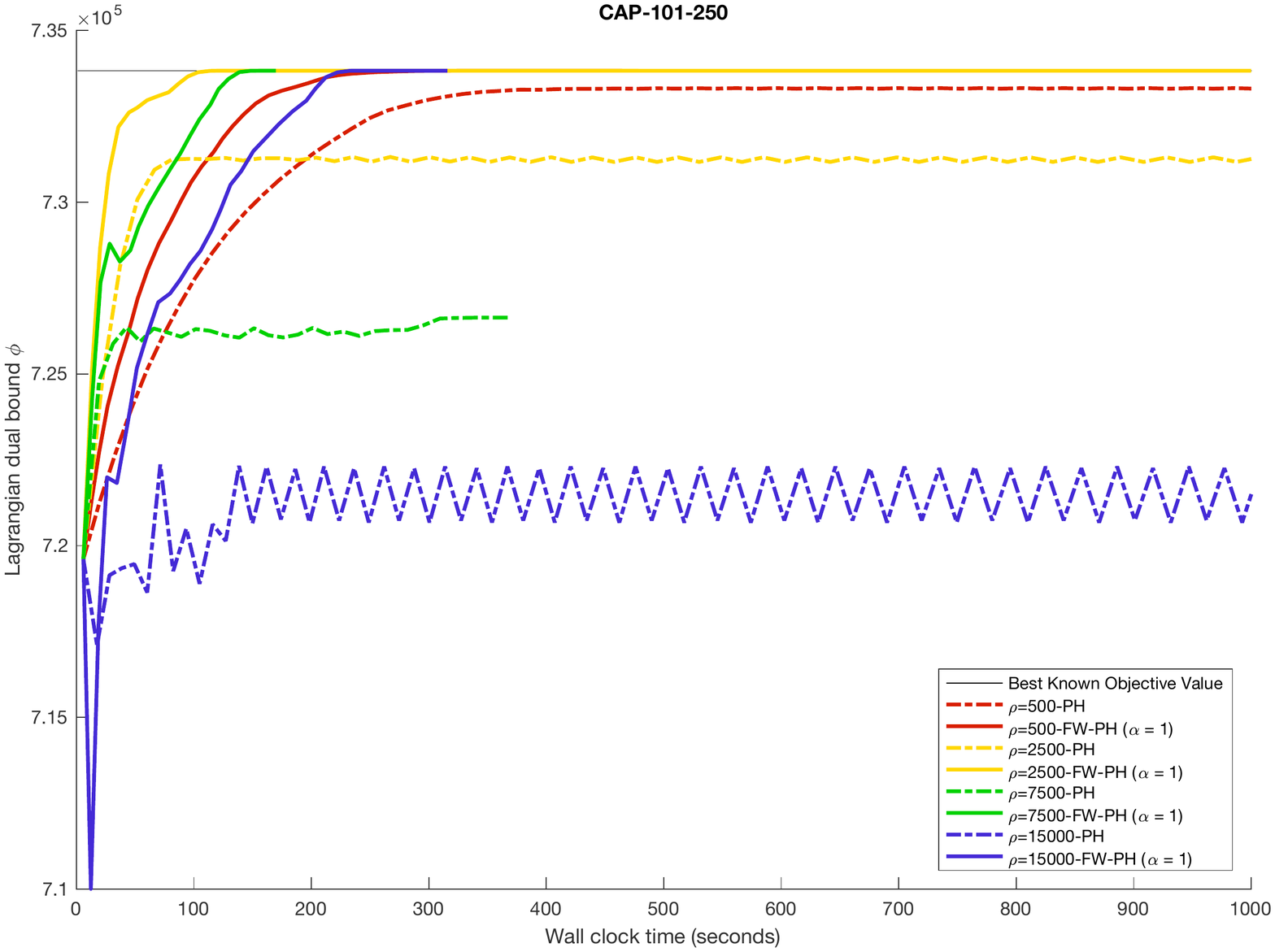}
\caption{Convergence profile for CAP-101-250 (PH and FW-PH with $\alpha = 1$)\label{FigCAP101A}}
\end{figure}
\begin{figure}[h]
\includegraphics[trim = 30mm 10mm 30mm 10mm, clip,width=0.77\textwidth]{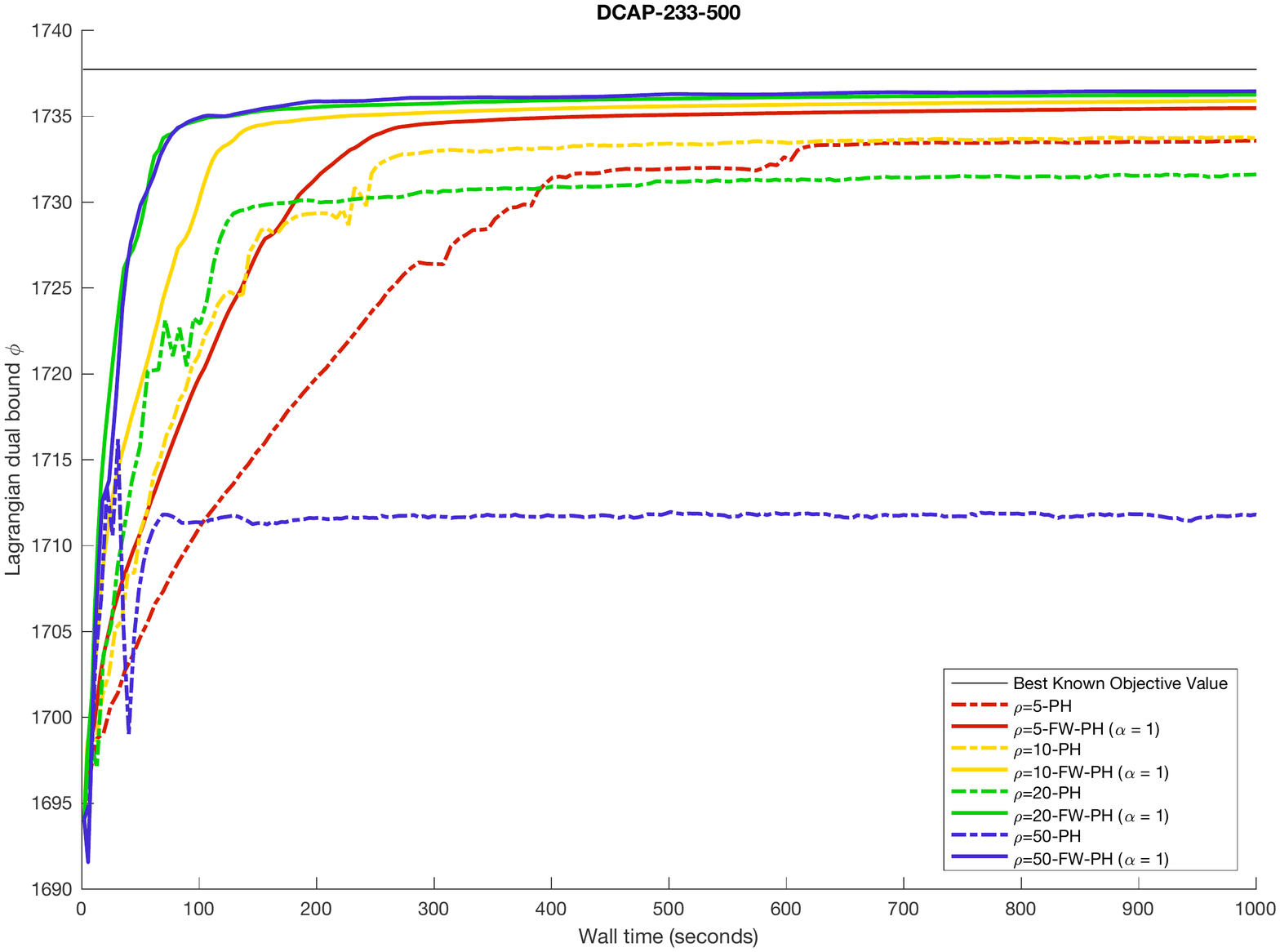}
\caption{Convergence profile for DCAP-233-500 (PH and FW-PH with $\alpha = 1$)\label{FigDCAP233500A}}
\end{figure}
\begin{figure}[btp]
\includegraphics[trim = 30mm 10mm 25mm 10mm, clip,width=0.85\textwidth]{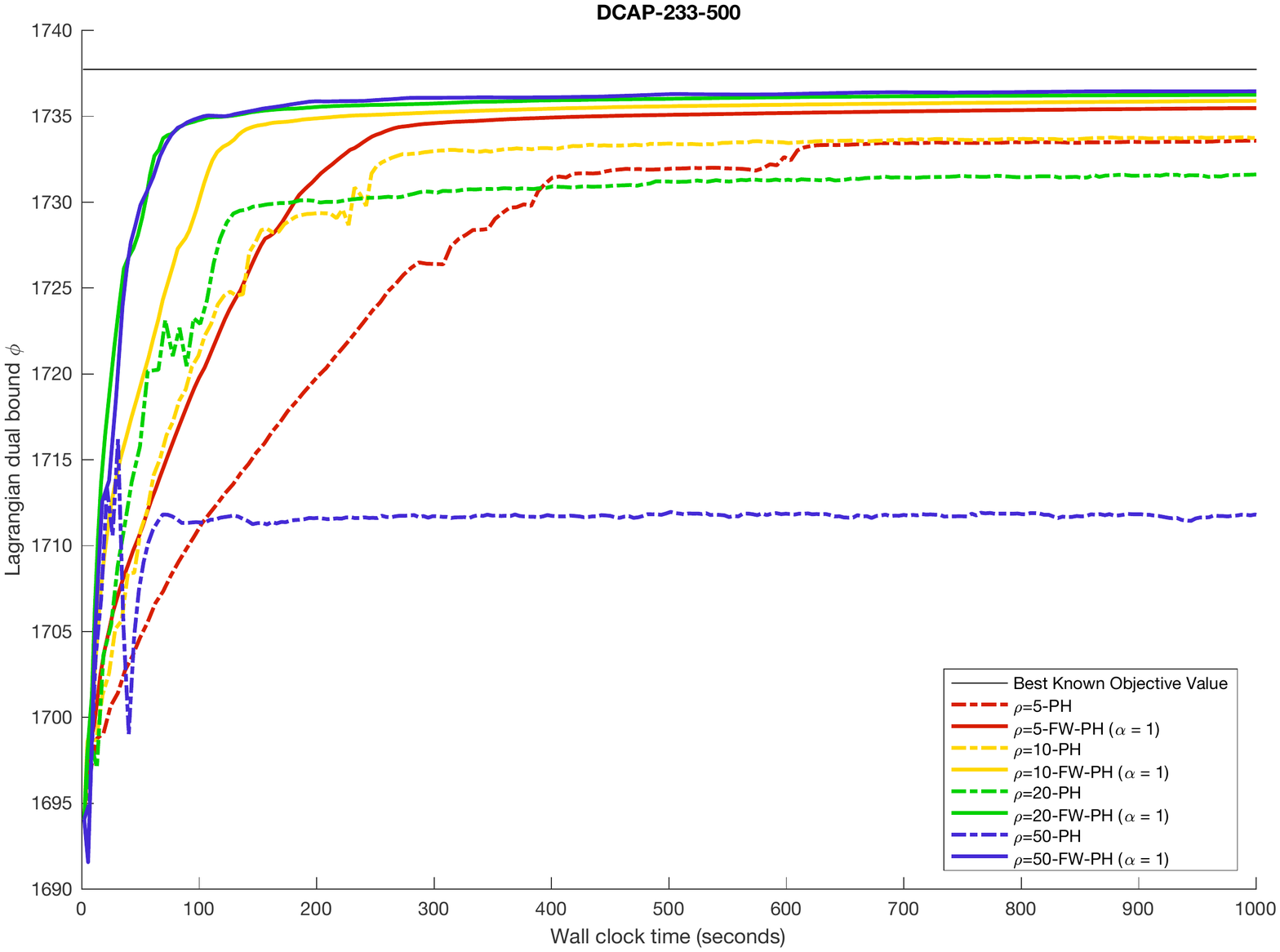}
\caption{Convergence profile for SSLP-5-25-50 (PH and FW-PH with $\alpha = 1$)\label{FigSSLP5-25-50A}}
\end{figure}
\begin{figure}[hbtp]
\includegraphics[trim = 30mm 10mm 30mm 10mm, clip,width=0.85\textwidth]{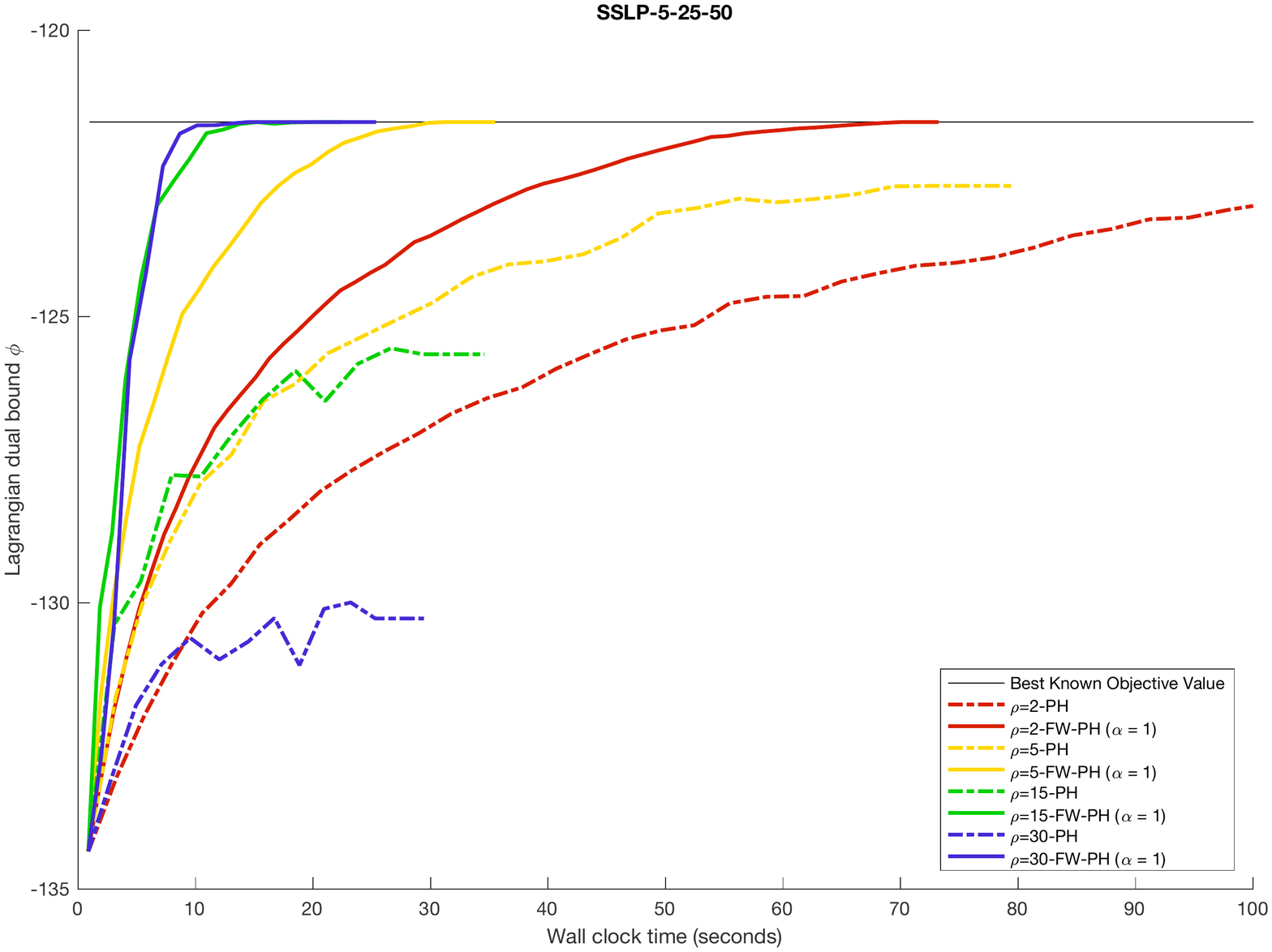}
\caption{Convergence profile for SSLP-10-50-100 (PH and FW-PH with $\alpha = 1$)\label{FigSSLP10-50-100A}}
\end{figure}

\end{document}